\documentclass[colorinlistoftodos,12pt,bibliography=totoc]{article}

\usepackage[utf8]{inputenc}

\usepackage[utf8]{inputenc}
\usepackage{amsmath,amsfonts,amsthm,epsfig,latexsym,graphicx,amssymb,xfrac,faktor}
\usepackage{enumerate} 
\usepackage[english]{babel}

\usepackage{bbm} 

\usepackage{mathtools}
\usepackage{stmaryrd} 

\usepackage[dvipsnames]{xcolor} 
\usepackage{tikz-cd}

\usepackage{graphpap} 
\usepackage[
  colorlinks = true,
  linkcolor = RoyalBlue, 
  citecolor = RoyalBlue 
]{hyperref}

\usepackage{sectsty} 
\sectionfont{\centering\color{RoyalBlue}} 
\subsectionfont{\centering\color{RoyalBlue}} 

\DeclareTextFontCommand{\emph}{\color{RoyalBlue}\em} 

\usepackage{ocgx2}
\usepackage{fancyhdr} 
\usepackage{theoremref}
\usepackage{xargs}     

\newcommand{\RR}{\mathbb{R}}
\newcommand{\NN}{\mathbb{N}}

\newcommand{\cC}{\mathcal{C}}
\newcommand{\ZZ}{\mathbb{Z}}
\newcommand{\QQ}{\mathbb{Q}}

\newcommand{\Class}{\mathcal{C}}

\newcommand{\bfrac}[2]{{
  \raisebox{0.1em}{\scalebox{0.85}{$#1$}}/
  \raisebox{-0.1em}{\scalebox{0.85}{$#2$}}
}}

\newcommand{\vertiii}[1]{{\left\vert\kern-0.25ex\left\vert\kern-0.25ex\left\vert #1 
    \right\vert\kern-0.25ex\right\vert\kern-0.25ex\right\vert}}

\newcommand{\wb}{\overline}
\newcommand{\wt}{\widetilde}
\newcommand{\wh}{\widehat}

\newcommand{\set}[1]{\left\{#1\right\}}

\newcommand{\lint}{\llbracket}
\newcommand{\rint}{\rrbracket}
\newcommand{\intint}[1]{{\lint #1 \rint}}        
\newcommand{\intintl}[1]{{\rint #1 \rint}}        
\newcommand{\intintr}[1]{{\lint #1 \lint}}        

\makeatletter
\def\@tvsp{\mathchoice{{}\mkern-4.5mu}{{}\mkern-4.5mu}{{}\mkern-2.5mu}{}}
\def\ltrivert{\left|\@tvsp\left|\@tvsp\left|}
\def\rtrivert{\right|\@tvsp\right|\@tvsp\right|}

\def\ldrivert{\left|\@tvsp\left|}
\def\rdrivert{\right|\@tvsp\right|}
\makeatother

\newcommand{\recto}{\rightsquigarrow}              
\newcommand{\rectot}{\leftrightsquigarrow}

\newcommand{\Rec}{\mathcal{R}}      

\newcommand{\MM}{\mathcal{M}}

\DeclareMathOperator{\cste}{cste}
\DeclareMathOperator{\im}{im}
\DeclareMathOperator{\len}{len}

\DeclareMathOperator{\QL}{QL}
\DeclareMathOperator{\diam}{diam}
\DeclareMathOperator{\rad}{rad}

\DeclareMathOperator{\e}{e}
\DeclareMathOperator{\clos}{clos}
\DeclareMathOperator{\conv}{conv}

\DeclareMathOperator{\Leb}{Leb}             

\newtheoremstyle{colorplain}%
{\topsep}   
{\topsep}   
{\itshape}  
{0pt}       
{} 
{.}         
{5pt plus 1pt minus 1pt} 
{\textbf{\textcolor{RoyalBlue}{\textbf{\thmname{#1} \thmnumber{#2}}}}\thmnote{ (#3)}}
{}

\newtheoremstyle{colorremark}%
{\topsep}   
{\topsep}   
{}  
{0pt}       
{\itshape} 
{.}         
{5pt plus 1pt minus 1pt} 
{\textcolor{RoyalBlue}{\thmname{#1} \thmnumber{#2}}\thmnote{ (#3)}}
{}

\newtheoremstyle{colordefinition}%
{\topsep}   
{\topsep}   
{}  
{0pt}       
{} 
{.}         
{5pt plus 1pt minus 1pt} 
{\textcolor{RoyalBlue}{\textbf{\thmname{#1} \thmnumber{#2}}}\thmnote{ (#3)}}
{}

\theoremstyle{colorplain}
\newtheorem{theorem}{Theorem}

\numberwithin{theorem}{section}

\newtheorem{maintheorem}{Theorem}

\newtheorem{remark}[theorem]{Remark}
\newtheorem{example}[theorem]{Example}

\newtheorem{lemma}[theorem]{Lemma}

\newtheorem{proposition}[theorem]{Proposition}

\newtheorem{corollary}[theorem]{Corollary}

\theoremstyle{colorremark}

\newtheorem{question}{Question}

\makeatletter 
\newenvironment{proofabstract}[1][\proofname]{
  \par
  \pushQED{\qed}%
  \normalfont \topsep6\p@\@plus6\p@\relax
  \trivlist
  \item\relax
  {\itshape
    #1\@addpunct{.}}\hspace\labelsep\ignorespaces
}{%
  \popQED\endtrivlist\@endpefalse
}

\renewenvironment{proof}[1][Proof]{
  \setcounter{claim}{0}
  \setcounter{claimproof}{0}
  \par
  \pushQED{\qed}%
  \normalfont\topsep6\p@\@plus6\p@\relax
  \trivlist
  \item\relax
  {\itshape\color{RoyalBlue}#1\@addpunct{.}}\hspace\labelsep\ignorespaces
}{%
  \popQED\endtrivlist\@endpefalse
}

\makeatother

\newcounter{claimproof} 

{\begin{proofabstract}[Sketch proof]}
    {\end{proofabstract}}

\theoremstyle{colordefinition}
\newtheorem{definition}[theorem]{Definition}

\title{Partial section I: $\alpha$-recurrence \\and equivariant Lyapunov maps}
\author{Théo Marty}
\date{}

\begin{document}

\maketitle
\begin{abstract}
	This is the first article in a series that aims at classifying partial sections of flows, that is a general family of transverse surfaces.
	In this part, we deal with the dynamical aspect of the question.
	
	Given a flow on a compact manifold~$M$ and a cohomology class~$\alpha$ of rank 1, we give a criterion for the existence of an $\alpha$-equivariant Lyapunov map on an Abelian covering of~$M$ associated to $\alpha$. 
	
	One important aspect of the existence of such Lyapunov maps, and of the classification of partial sections, is a type of recurrence set relative to $\alpha$. We describe how that set depends on~$\alpha$.
\end{abstract}

\section*{Introduction}
\addcontentsline{toc}{section}{Introduction}

Surfaces transverse to a flow received renewed attention in recent years, to help characterize dynamical and topological properties of flows. Transverse surfaces exist with several flavors: transverse everywhere, transverse except on the boundary, that intersect every flow line, or not. The most well-understood are \emph{global cross-sections}: compact hyper-surfaces transverse to the flow and that intersect every flow line. Fried \cite{Fried82} classified the set of global cross-sections up to isotopy along the flow, using a cohomological criterion\footnote{Schwartzman \cite{Schwartzman1957} and Sullivan \cite{Sullivan1976} contributed to the classification, but Fried's work is the most complete of the three.}. 
\emph{Partial cross-sections} are, similarly, compact hyper-surface transverse to the flow, with no assumption regarding the intersection with all orbits. Only partial results have been shown for partial cross-sections, and only in restricted contexts. We refer to \cite{Mosher1989,Mosher1990} and \cite{Landry2024} for articles on the matter, for Anosov and pseudo-Anosov flows in dimension~3.

We aim to give a general picture of partial cross-sections\footnote{Partial Birkhoff-sections (with boundary tangent to the flow) will be considered in a later entry in the series.} in a series of articles. As this is the first, we describe here only the dynamical part of that question.

We fix a compact topological manifold~$M$ and a continuous flow $\varphi$ on~$M$. 
Let us summarize some of Fried's ideas. Take a global cross-section $S$ of the flow $\varphi$. It can be represented as a level set of a smooth fibration $f\colon M\to\bfrac{\RR}{\ZZ}$ over the circle, so that additionally, all level sets are global cross-sections. In particular $f$ increases along the flow. Fried defined a compact subset $D^F_\varphi$ in $H_1(M,\RR)$, that captures the asymptotic homological direction of very long orbits. Since $f$ is increasing along the flow, we have $df(D^F_\varphi)>0$. 

Conversely, take a smooth map $f\colon M\to\bfrac{\RR}{\ZZ}$ that satisfies $df(D^F_\varphi)>0$. One can average $f$ along the flow, to transform the homological criterion into a topological condition. That is if $g$ is the average of $f$ along the flow for a long enough time, then any level set of $g$ is a global cross-section. With a little more work, he proved that any two global cross-sections that are homologous are isotopic along the flow. So global cross-sections are parametrized by a subset of the cohomological module $H^1(M,\ZZ)$.

For partial cross-sections, the picture is more complicated, one needs to consider the weaker condition $df(D^F_\varphi)\geq 0$. We know that under this condition, the existence of a partial cross-section homologous to $df$ is not guaranteed in general, and two partial cross-sections may be homologous but not isotopic along the flow. So the general picture is more complicated. To parametrize partial cross-sections, we go back to the idea of finding a good map $f\colon M\to\bfrac{\RR}{\ZZ}$. Some level sets of $f$ are partial cross-sections if $f$ satisfies a Lyapunov-like property. 
To express that property, we lift $f$ to a good $\ZZ$-covering of~$M$.
Let $\alpha$ be in $H^1(M,\ZZ)$ and non-zero, representing $-df$. There is a natural $\ZZ$-covering space of~$M$ associated to $\alpha$, denoted by $\wh M_\alpha$. Note that $\varphi$ lifts to a flow on $\wh M_\alpha$. 


\begin{question}
	When is there a $\ZZ$-equivariant and Lyapunov map $h\colon\wh M_\alpha\to\RR$?
\end{question}

We say that $\alpha$ is \emph{quasi-Lyapunov} if, roughly speaking, every pseudo-orbit~$\gamma$ with small jumps satisfies $\alpha(\gamma)\leq 0$.


\begin{maintheorem}[Equivariant spectral decomposition]\label{mainthm-spectral-decomp}
	Let $\alpha$ be in $H^1(M,\ZZ)$.
	Then there exists a $\ZZ$-equivariant Lyapunov map $h\colon\wh M_\alpha\to\RR$ if and only if $\alpha$ is quasi-Lyapunov.
\end{maintheorem}

We state a stronger version in Theorem~\ref{thm-spectral-decomp}.
The proof utilizes Conley's arguments, adapted to build equivariant Lyapunov functions. To classify partial cross-sections, we mostly need this theorem and a variation of it.

\vline

It is well known that global cross-sections of a flow span a convex open cone in cohomology, which is related to the unit ball of the Thurston norm of the ambient manifold. We upgrade this by characterizing the set of cohomology class of partial cross-sections. 
For that, we characterize quasi-Lyapunov classes with real coefficients too. 

Let $\alpha$ be in $H^1(M,\RR)$, not necessarily with integer coefficients. Similarly, it corresponds an Abelian covering $\wh M_\alpha$ over~$M$, not a $\ZZ$-covering in general. 
Denote by $\Rec_\alpha$ the image in~$M$ of the recurrent set on $\wh M_\alpha$, which we call the $\alpha$-recurrent set. 
It plays a crucial role when understanding Lyapunov functions and thus partial cross-sections. Therefore, it is important to characterize it here.

Denote by $\QL_\varphi$ the set of quasi-Lyapunov classes. It is a convex cone, not necessarily closed nor polyhedral. For simplicity, we assume in the introduction that $\QL_\varphi$ is polyhedral and closed. As a polyhedral cone, it admits faces of various dimensions.

\begin{maintheorem}\label{mainth-dependence-arec}
	Assume that $\QL_\varphi$ is closed and polyhedral. For every face $F$ of $\QL_\varphi$, there exists a $\varphi$-invariant compact subset $\Rec_F\subset M$ that satisfies the following. For every quasi-Lyapunov class $\alpha$ in the interior of $F$, we have $\Rec_\alpha=\Rec_F$. Additionally, if $F'$ is a proper sub-face of $F$, then $\Rec_F$ is a proper subset of $\Rec_{F'}$. 
\end{maintheorem}

The general statement can be found in Theorem~\ref{thm-dep-face}, without the assumptions on $\QL_\varphi$.
Our last main result is an extension of Theorem~\ref{mainthm-spectral-decomp} to the real coefficients case. In the real coefficients case, $H_1(M,\ZZ)$ acts on $\wh M_\alpha$ though the deck transformation, and on $\RR$ through $\alpha$. That is we have $\delta\cdot t=t+\alpha(\delta)$ for any $t$ in $\RR$ and $\delta$ in $H_1(M,\ZZ)$. An equivariant map for these actions is said \emph{$\alpha$-equivariant}. 

\begin{maintheorem}[Equivariant spectral decomposition, real case]\label{mainthm-strongL-rat}
	Let~$M$ be a compact manifold, $\varphi$ be a continuous flow on~$M$ and $\alpha$ be in $H^1(M,\RR)$.
	If $\alpha$ is quasi-Lyapunov, then there exists a Lyapunov and $\alpha$-equivariant map $f\colon\wh M_\alpha\to\RR$..
\end{maintheorem}

We state a stronger version in Theorem~\ref{thm-quasiL-map-non-rat}. We ask the following question, which we could not answer.

\begin{question}
	Is there a form of converse to Theorem~\ref{mainthm-strongL-rat}?
\end{question}

The Theorems~\ref{mainthm-spectral-decomp} and~\ref{mainthm-strongL-rat} are stated with no regularity. Their stronger counter-parts add a conclusion: when $M$ is smooth and $\varphi$ satisfies a minimal regularity assumption, then the Lyapunov maps $f$ can additionally be taken smooth and satisfying $df(X)<0$ outside the recurrent set.

Let us describe the structure of the article. In Section~\ref{sec-preliminary}, we introduce our notation and conventions about Abelian covering and about Conley's formalism. Then we define a set similar to $D^F_\varphi$. The set $D^F_\varphi$ by Fried does not really fit our needs, so we define a second set $D_\varphi$, called the set of asymptotic pseudo-directions. The definition can be found in Section~\ref{sec-ass-ps-dir}. We prove in Appendix~\ref{app-asymptotic-cone} that the two sets span the same convex set. So they may essentially be exchange in many applications.

Section~\ref{sec-reduction-pso} contains several technical results, about reducing pseudo-orbits to shorter ones. 

In Section~\ref{sec-spectral-decomposition}, we introduce the two important notions of quasi-Lyapunov maps and $\alpha$-recurrence. We give several characterizations of these notions. 

Section~\ref{sec-QL-int-coeff} provides the main steps toward the classification of partial cross-sections. We prove the equivariant spectral decomposition in the integer case. We give a second result for a relaxed version of Lyapunov map. A \emph{pre-Lyapunov} map is a map that is non-increasing along the flow, constant on recurrent chains, plus one condition. We prove that when $\alpha$ is quasi-Lyapunov, there exist  $\alpha$-equivariant pre-Lyapunov maps whose restriction to the recurrent set, roughly speaking, can be any non-forbidden $\alpha$-equivariant function with values in $\ZZ$.

In Section~\ref{sec-general-quasiL}, we study quasi-Lyapunov classes with real coefficients. Quasi-Lyapunov classes with real coefficients are more delicate to work with, because we lose some compactness arguments. To compensate, we approximated any quasi-Lyapunov by ones with rational coefficients, and that satisfies similar dynamical properties. These approximations lead us to a proof of Theorem~\ref{mainth-dependence-arec} in Section~\ref{sec-dependence-reca}, and a quite technical proof of Theorem~\ref{mainthm-strongL-rat} in Section~\ref{sec-spect-decomp-real}.

\paragraph{Acknowledgments.} I thank Christian Bonatti for the precisions on Conley's theory, among other things. I thank Pierre Dehornoy for the discussions that motivated this project. The research was conducted during my stay at the Université de Bourgogne, in Dijon, funded in honor of Marco Brunella.
\newpage

{
	\hypersetup{linkcolor=black}
	\tableofcontents \label{ToC}
}

\section{Preliminary} \label{sec-preliminary}

In the article,~$M$ is a topological manifold, assumed connected, compact, possibly with boundary, and equipped with a compatible metric. Most of our arguments are valid without assuming much regularity on~$M$.

We fix a continuous flow $\varphi_t\colon M\to M$. All flows are assumed complete, that is defined at all time. In particular, $\varphi$ preserves $\partial M$. We allow~$M$ to have a non-empty boundary, so that the theory can be applied to manifolds obtained after blowing-up finitely many closed curves.


\subsection{Abelian covering}\label{sec-Z-covering}

Fried \cite{Fried82} showed the strong connection between global cross-sections, cohomology classes (of rank 1) and $\ZZ$-covers. 

Let us denote by $\wh M\xrightarrow{\wh\pi}M$ the universal Abelian covering of~$M$, given by the quotient of the universal covering of~$M$ under the action of the kernel of the morphism $\pi_1(M)\to H_1(M,\ZZ)$. The module $H_1(M,\ZZ)$ acts on $\wh M$ as the deck transformation of the covering map. We denote by $\delta\cdot x$ the image of~$x$ in $\wh M$ under the action of an element $\delta$ in $H_1(M,\ZZ)$.

Let $\alpha$ be in $H^1(M,\RR)$. 
We denote by $\ker_\ZZ(\alpha)\subset H_1(M,\ZZ)$ the kernel of $\alpha$ inside $H_1(M,\ZZ)$, that is the set of $\delta$ in $H_1(M,\ZZ)$ that satisfies $\alpha(\delta)=0$. We define the Abelian covering map
$\wh M_\alpha\xrightarrow{\pi_\alpha}M$
as the quotient of $\wh M\to M$ by the action of $\ker_\ZZ(\alpha)$. Its deck transformation is $\bfrac{H_1(M,\ZZ)}{\ker_\ZZ(\alpha)}$. 
We call it the \emph{covering associated to $\alpha$}. 
It is characterized as follows: a closed curve $\gamma\subset M$ lifts to a closed curve in $\wh M_\alpha$ if and only if we have $\alpha(\gamma)=0$.

The flow $\varphi$ lifts to two flows on $\wh M$ and $\wh M_\alpha$.

Recall that $H^1(M,\ZZ)$ can be identified as a subset of $H^1(M,\RR)$. 
We say that $\alpha$ has integer coefficient if it lies inside $H^1(M,\ZZ)$.
When $\alpha$ has integer coefficients and is not zero, $\bfrac{H_1(M,\ZZ)}{\ker_\ZZ(\alpha)}$ is naturally isomorphic to~$\ZZ$, so $\wh M_\alpha$ is a $\ZZ$-covering.

\paragraph{Descripiton of the rank 1 cohomology.} The cohomology set $H^1(M,\ZZ)$ plays a central role in this article. 
It is well-known that $H^1(M,\ZZ)$ is isomorphic (as a group) to the set $[M,\bfrac{\RR}{\ZZ}]$ of homotopy classes of continuous maps from~$M$ to $\bfrac{\RR}{\ZZ}$. The isomorphism can be though as follows. Given a smooth map $f\colon M\to\bfrac{\RR}{\ZZ}$, its differential $df=f^*dt$ is a closed 1-form on~$M$ (not exact in general), so it induces a cohomology class $[df]$. Note that $[df]$ paired with a closed curve $\gamma\subset M$ is equal to the number of turns made by $f\circ\gamma$ inside the circle $\bfrac{\RR}{\ZZ}$, so is an integer. Thus, $[df]$ has integer coefficient.
Given a class $\alpha$ in $H_1(M,\ZZ)$ and a continuous map $f\colon M\to\bfrac{\RR}{\ZZ}$, the map $f$ is \emph{cohomologous to~$\alpha$} when there exists a smooth map $g$ homotopic to $f$ which satisfies $[dg]=\alpha$. 

\paragraph{Equivariance.} 
A function $f\colon\wh M\to\RR$ or $f\colon\wh M_\alpha\to\RR$ is said \emph{$\alpha$-equivariant} if for all~$x$ and $\delta$ in $H_1(M,\ZZ)$, we have $f(\delta\cdot x)=\alpha(\delta)+f(x)$. 
When $\alpha$ has integer coefficients, $f$ is $\alpha$-equivariant if and only it projects down on~$M$ to a function $h\colon M\to\bfrac{\RR}{\ZZ}$ cohomologous to~$\alpha$, that is with  $h\circ\pi_\alpha= f\pmod 1$.

Assume that $\alpha$ has integer coefficients, and write $\alpha=n\beta$ for some $n\geq 1$ and some primitive class $\beta$. Then a map $f\colon\wh M_\alpha\to\RR$ is $\alpha$-equivariant if and only if $\tfrac{1}{n}f$ is $\ZZ$-equivariant. 
For simplicity, we used the $\ZZ$-equivariance formulation in Theorem~\ref{mainthm-spectral-decomp}, but we will not use it again.


\paragraph{Compactification.}
When $\wh M_\alpha\to M$ is a $\ZZ$-covering, it is convenient to compactify $\wh M_\alpha$  with two points $\pm\infty$. We denote by $\wb M_\alpha=\wh M_\alpha\cup\{-\infty,+\infty\}$ the topological space with the following neighborhood of the points $\pm\infty$. 
Given a compact subset $K\subset\wh M_\alpha$ for which $\pi_\alpha(K)=M$ holds, the sets $\{-\infty\}\bigcup_{n\leq 0}n\cdot K$ and $\{+\infty\}\bigcup_{n\geq0}n\cdot K$ are closed neighborhoods of respectively $-\infty$ and $+\infty$. We extend the flow to $\wb M_\alpha$ to be constant on $\pm\infty$. Note that for any point in $\wh M_\alpha$, the orbit $n\cdot x$, for~$n$ in $\ZZ$, goes to $\pm\infty$ when~$n$ goes to $\pm\infty$.

\paragraph*{Choice of metric.} The soon introduced notion of~$\epsilon$-pseudo-orbit depends on a choice of metric. The specific choice of metric usually does not matter, but taking a good metric makes the argument easier.

Let $N$ be a topological space, and $d,d'$ be two metrics on $N$ compatible with its equipped topology. We say that $d$ and $d'$ are \emph{comparable at small scale} if for all $\epsilon>0$, there exists $\epsilon'>0$ so that every $\epsilon'$-ball for one metric is included in an~$\epsilon$-ball for the other metric. When $N$ is compact, any two compatible metrics are comparable at small scale. 

For the compact manifold~$M$, choose any metric compatible with the topology, denoted here by $d_M$.
When $N$ is given by a covering space $N\xrightarrow{\pi}M$, we lift the metric from~$M$ to $N$ as follows. 
We call the \emph{radius of injectivity} of~$M$, denoted by $\diam(M)$, the supremum of the $\eta>0$ that satisfies that all balls of radius $\eta$ belong in a contractible open subset of~$M$. We choose a metric $d_N$ on $N$ which is invariant by the deck transformation of $\pi$, and so that for any ball $B\subset N$ of radius less than $\rad(M)$, the projection $B\xrightarrow{\pi}M$ is an isometry onto its image. Such a metric can be build for instance by taking chains of points in $N$, at distance less than $\rad(M)$, and summing the distance between successive points. Any two such metrics are comparable at small scale.

For the compactification $\wb M_\alpha$ any two metrics are comparable at small scale. So we chose any of them.

\subsection{Pseudo-orbit and recurrence chains}\label{sec-pso}

In this section, $N$ is a metric space of one of these two types: either $N$ equal~$M$, $N$ is a covering space over~$M$, with the induced metric, or~$N$ is the compactification $\wb M_\alpha$ of $\wh M_\alpha$ given in the previous subsection. Let $\psi\colon\RR\times N\to N$ be a continuous flow on $N$. If $N$ is equal to $\wb M_\alpha$, $\psi$ is a usual flow on $\wh M_\alpha$, and is constant on $\{+\infty,-\infty\}$. We fix some $T>0$.
an \emph{$(\epsilon,T)$-pseudo-orbit} is a curve $\gamma\colon[0,l[\to N$, for some $l>0$, that satisfies the following:
\begin{itemize}
	\item $\gamma$ is right-continuous,
	\item $\gamma$ is continuous outside a finite set $\Delta$ that is disjoint from $[0,T[$,
	\item the points in $\Delta$ are distant by at least $T$,
	\item for any $t,s$ for which $\gamma$ is continuous on $[t,s]$, we have $\gamma(s)=\psi_{s-t}\circ\gamma(t)$,
	\item for any $t$ in $\Delta$, the distance between $\gamma(t)$ and $\lim_{s\to t^-}\gamma(s)$ is smaller than~$\epsilon$.
\end{itemize}

A \emph{periodic $(\epsilon,T)$-pseudo-orbit} is similarly a function $\gamma\colon\bfrac{\RR}{l\ZZ}\to M$, with $l>0$, that satisfies the same four conditions. We fix some $T>0$ and we will, most of the time, simply write~$\epsilon$-pseudo-orbit instead of $(\epsilon,T)$-pseudo-orbit.

Note that any orbit arc, of positive length, is an~$\epsilon$-pseudo-orbit. Similarly, any periodic orbit is a periodic~$\epsilon$-pseudo-orbit.
From now on, we write $\gamma(u^-)=\lim_{t\to u^-}\gamma(t)$.
At a discontinuity time $u$ in $\Delta$, we say that $\gamma$ jumps from $\gamma(u^-)$ to $\gamma(u)$.
The scalar $l$ is called the \emph{length} of the pseudo-orbit, later denoted by $\len(\gamma)$. Additionally, using $x=\gamma(0)$ and $y=\gamma(\len(\gamma)^-)$, we say that $\gamma$ goes from the point~$x$ to the point $y$.

\begin{definition}
	Let $x,y$ be two points in $N$. We write $x\recto y$ if for all $\epsilon>0$, there exists an~$\epsilon$-pseudo-orbit from~$x$ to $y$. 
\end{definition}

It does not depend on the value of $T>0$ (once fixed), not on the specific choice of metric on~$M$ and $\wb M_\alpha$. 

A point~$x$ in $N$ is said \emph{recurrent} when we have $x\recto x$. The set of recurrent point is called the \emph{recurrent set}. We denote by $\Rec$ the recurrent set on the manifold~$M$. Given $\alpha$ in $H^1(M,\RR)$, we denote by $\wh\Rec_\alpha$ the recurrent set on~$\wh M_\alpha$. The relation $x\rectot y$, given by $x\recto y$ and $y\recto x$, is an equivalence relation. Its equivalence classes are called \emph{recurrence chain}. Similarly, for two recurrence chains $R,R'$, we denote by $R\recto R'$ if we have $x\recto y$ for some/all~$x$ in~$R$ and $y$ in $R'$. 

\begin{remark}
	It is well known that both the recurrence set and the recurrence chains are $\psi$-invariant and closed inside $N$.
\end{remark}

\subsection{Lyapunov functions}\label{sec-L-function}

Conley \cite{Conley1978,Conley1988} developed the notion of (complete) Lyapunov maps in dynamical systems. 
Let us recall these notions and introduce a variation of them.

A continuous map $f\colon N\to\RR$ is \emph{Lyapunov} if the three following properties are satisfied: $f$ is constant on every recurrence chain, two distinct recurrence chains have distinct values, and $f$ is decreasing along the flow outside the recurrent set. 


\begin{theorem}[Conley \cite{Conley1978}]\label{thm-Conley-L}
	Let~$M$ be a compact manifold. Any continuous flow on~$M$ admits a Lyapunov function.
\end{theorem}

We introduce the variations of this notion.
A continuous map $f\colon N\to\RR$ is called \emph{pre-Lyapunov} if for any $x,y$ in $N$, $x\recto y$ implies $f(x)\geq f(y)$. It implies that $f$ is constant on each recurrence chain. The function $f$ is called \emph{strongly-Lyapunov} if it is pre-Lyapunov, if distinct recurrence chains have distinct values, and if additionally $x\recto y$ and $y\not\recto x$ imply $f(x)>f(y)$. 

We define pre-Lyapunov functions for flexibility reasons, and strongly-Lyapunov because we will work in a non-compact setting. We believe that being "non-increasing" for the relation $\recto$ is a valid enough property to add it to the conclusion of our results.


\begin{lemma}\label{lem-L-def}
	A strongly-Lyapunov function is Lyapunov. The converse holds true when $N$ is compact.

	Take~$M$ compact, $\alpha$ in $H^1(M,\ZZ)$ and $N=\wh M_\alpha$. Then any $\alpha$-equivariant and Lyapunov map $f\colon\wh M_\alpha\to\RR$ is strongly-Lyapunov.
\end{lemma}

The first claim follows from the definitions.
The converse is not true in general, without assumption on $N$, as illustrated in Example~\ref{ex-L-no-sL}. 
We prove Lemma~\ref{lem-L-def} after Lemma~\ref{lem-blocking-box}. 

Let us recall some important tools that were developed to prove Theorem~\ref{thm-Conley-L}.
Given a subset $X\subset N$, denote by $\omega(X)$ the \emph{$\omega$-limit} of $X$ which is defined by 
$$\omega(X)=\bigcap_{t\geq 0}\wb{\psi_{[t,+\infty[}(X)}.$$
Similarly, the \emph{$\alpha$-limit} of $X$ is the $\omega$-limit of the time-reversed flow.
A non-empty subset $A\subset N$, is said to be an \emph{attractor} of $\psi$ if $A$ is compact, $\psi$-invariant and if there exists a neighborhood $U$ of $A$ that satisfies $\omega(U)=A$. A \emph{repeller} is an attractor of the time-reversed flow. Given an attractor $A$, we denote its domain of attraction by $D(A)=\{x\in N,\omega(x)\subset A\}$. The set $A^*=N\setminus D(A)$ is a repeller. We call $A,A^*$ a pair of attractor-repeller.

There are at most countably many attractors, and for any pair of recurrence chains $R_1,R_2$ with $R_1\not\recto R_2$, there exists an attractor $A$ that satisfies $R_1\subset A$ and $R_2\subset A^*$. So pairs of attractor-repellers can be used to recover individual recurrence chains.

The following lemma is a key technical tool that we will use several times. It is a variation of well-known result of Conley \cite{Conley1978}.

\begin{lemma}\label{lem-preL-atr-rep}
	Let $\alpha$ be in $H^1(M,\ZZ)$, and $A,A^*$ be a pair of attractor/repeller in $\wb M_\alpha$, $U,U^*$ be neighborhoods of respectively $A$ and $A^*$. Then there exists a pre-Lyapunov function $f\colon\wb M_\alpha\to[0,1]$ which satisfies the following properties:
	\begin{itemize}
		\item $f\equiv 0$ on a neighborhood of $A$,
		\item $f\equiv 1$ on a neighborhood of $A^*$,
		\item $f$ is decreasing along the flow outside $U\cup U^*$.
	\end{itemize}
	When~$M$ is smooth and $\psi$ is generated by a continuous and uniquely integrable vector field $X$, $f$ can be taken smooth, and so that $df(X)<0$ holds outside~$U\cup U^*$.
\end{lemma}

Traditionally, $f$ is chosen decreasing along the flow outside $A\cup A^*$. We choose a different convention: $f$ is constant on a neighborhood of $A\cup A^*$, and decreasing along the flow outside. The reason is that we will consider series of these functions, and our convention makes the convergence trivial.

Though the result is very classic, it is not stated in this form, so let us give us a short proof.

Given a continuous map $f\colon N\to\RR$ and a continuous map $h\colon\RR\to\RR$ that decreases fast enough to zero, we denote by $f*_\varphi h\colon N\to\RR$ (when well-defined) the \emph{convolution along the flow} between $f$ and $h$, defined by $f*_\varphi(x)=\int_\RR f\circ\varphi_s(x)h(t-s)ds$. 

\begin{proof}
	Conley \cite{Conley1978} proved the existence of a continuous map $g\colon\wb M_\alpha\to\RR$, with $g^{-1}(0)<A$ and $g^{-1}(1)=A^*$ and that is decreasing outside $A\cup A^*$. e $g$ smooth and satisfying $dg(X)<0$ outside $A\cup A^*$.

	Take some small $\epsilon>0$ and a smooth map $h\colon[0,1]\to[0,1]$ that satisfies $h([0,\epsilon])=0$, $h([1-\epsilon,1])=1$, $h$ is increasing on $[\epsilon,1-\epsilon]$, and $dh(X)<0$ holds on $]\epsilon,1-\epsilon[$ under the regulatity assumption. Then, when~$\epsilon$ is small enough, $f(M\setminus(U\cup U^*))$ lies in $]\epsilon,1-\epsilon[$. So the map $f=h\circ g$ satisfies the conclusion of the lemma, apart from the regularity.

	We now take the regularity assumptions on~$M$ and $\varphi$. Up to replacing $g$ by the convolution along the flow of $g$ by a Gaussian map, we may assume that $g$ is of class $\Class^1$ along the flow, and that $dg(X)<0$ holds outside $A\cup A^*$.	
	Then, using Corollary 5.7 in~\cite{Fathi2019}, we may assume that $g$ is smooth outside $A\cup A^*$. Then $f=h\circ g$ is smooth everywhere, and satisfies $df(X)<0$ outside~$U\cup U^*$.
\end{proof}

Denote by $\Rec^*$ the set of recurrence chains of $\psi$. We view $\Rec^*$ as the quotient of $\Rec$ by the equivalence relation $x\rectot y$, and equip $\Rec^*$ with the quotient topology. The previous lemma has a well-known consequence:

\begin{corollary}\label{cor-rec-disconnect}
	The set $\Rec^*$ is compact and totally disconnected.
\end{corollary}

\begin{proof}
	The set $\Rec^*$ is compact as the quotient of a compact space. Take two distinct recurrence chains $R_1,R_2$. Then there exists an attractor $A$ that contains one $R_i$ but not the other one. Let $f\colon N\to[0,1]$ be the pre-Lyapunov function given by Lemma~\ref{lem-preL-atr-rep}. The two sets $\Rec\cap f^{-1}([0,\tfrac{1}{2}[)$ and $\Rec\cap f^{-1}(]\tfrac{1}{2},1])$ yield a partition of $\Rec^*$ in two open sets, each of them contains one of the sets $R_i$. Thus, $\Rec^*$ is totally disconnected.
\end{proof}

We now prove that our definition of strongly-Lyapunov coincide with the standard definition in the compact case.

\begin{lemma}\label{lem-blocking-box}
	Assume $N$ compact. Let $f\colon N\to\RR$ be a continuous function, non-increasing along the flow lines and constant on each recurrence chain. Take $r$ outside $f(\Rec)$. Then for any $x,y$ in $N$ with $f(x)>r\geq f(y)$, we have $y\not\recto x$. 
\end{lemma}

Although the result is well-known for experts, we express it in a flexible way to apply it for pre-Lyapunov-like functions in the future.



\begin{proof}[Proof of Lemma~\ref{lem-blocking-box}]
	Let $f,r,x,y$ be as in the lemma. 
	We take a Gaussian map $h$ given by $h(t)=\sqrt{\tfrac{a}{\pi}}\e^{-at^2}$ for some $a>0$, and $S=(f*_\varphi h)^{-1}(r)$. Note that $f*_\varphi h$ and $f$ coincide on the recurrent set, so not point in $S$ is recurrent.
	For every point~$z$ in $S$, its orbit accumulates on two distinct recurrence chains $R^-,R^+$ respectively in the past and in the future. Then, we have $f(R^-)>f(z)>f(R^+)$. Thus, $f$ is not constant on $\varphi_\RR(z)$. It follows from standard convolution arguments that $f*_\varphi h$ is decreasing along the flow on any orbit that intersects $S$. 
	Therefore, $S$ is a topological hyper-surface, compact, and transverse to the flow. 
	
	Take some $T'>0$ smaller than $T$. Denote by $E$ and $K$ the compact sets $E=(f*_\varphi h)^{-1}([-\infty,r[)$ and $K=\varphi_{[-T',T']}(S)$.
	Then $E\cup K$ is a neighborhood of $E$ inside $N$, so there exists $\epsilon>0$ for which the~$\epsilon$-neighborhood of $E$ lies inside $E\cup K$. The image $\varphi_T(E\cup K)$ lies inside $E$. It follows that every~$\epsilon$-pseudo-orbits that starts in $E$ remains in $E\cup K$. 
	
	Up to taking $a$ larger, we may assume that $f*_\varphi h(x)>r$ holds. And up to taking $T'$ smaller, we may assume that $K$ does not contain~$x$. Then~$x$ lies outside $E\cup K$, so there is no~$\epsilon$-pseudo-orbit from $y$ to~$x$. Therefore, we have $y\not\recto x$.
\end{proof}

We can now prove that under a compactness assumption, a Lyapunov function is strongly-Lyapunov.

\begin{proof}[Proof of Lemma \ref{lem-L-def}]
	Take $\alpha$ in $H^1(M,\ZZ)$, $N=\wh M_\alpha$ and $f\colon\wh M_\alpha\to\RR$ an $\alpha$-equivariant and Lyapunov function. It contains the first case for $\alpha=0$. 
	In the case $\alpha\neq0$, notice that $f$ converges toward $\pm\infty$ in $\pm\infty$. So we can extend $f$ as a continuous map $\wb f\colon\wb M_\alpha\to\wb\RR$. Denote by~$\wb\Rec_\alpha$ the recurrent set on $\wb M_\alpha$.
	We claim that $\wb f$ is Lyapunov.

	Let $x,y$ in $\wh M_\alpha$ be two points and assume that $f(x)<f(y)$ holds.
	Corollary~\ref{cor-rec-disconnect} implies that $\wb f(\wb\Rec_\alpha)$ is totally disconnected, and so its interior is empty.
	Hence, there exists some $r$ in $]f(x),f(y)[$ that lies outside $\wb f(\wb\Rec_\alpha)$. 
	Then Lemma~\ref{lem-blocking-box} implies that, for the flow on $\wb M_\alpha$, we have $x\not\recto y$. It implies $x\not\recto y$ for the flow on $\wh M_\alpha$ too.
	Therefore, by contraposition, $x\recto y$ implies $f(x)\geq f(y)$.

	We now assume that $x\recto y$ and $y\not\recto x$ hold true. We consider several cases on $x,y$ to prove the inequality $f(x)>f(y)$. In the first case,~$x$ is not recurrent and $y$ lies in the positive orbit of $y$. Then $f(x)>f(y)$ follows from the decreasing along the flow hypothesis. In the second case,~$x$ is not recurrent and $y$ is not in the same orbit. Then we have $x\recto\varphi_1(x)\recto y$ and $f(x)>f\circ\varphi_1(x)\geq f(y)$ follows from the two previous arguments. In the third case,~$x$ and $y$ are recurrent. Then~$x$ and $y$ belong to two distinct recurrent chains, which have distinct values, so $f(x)>f(y)$ holds. The last case, that is~$x$ is recurrent and $y$ is not recurrent, is similar to the second case by reversing the flow. It follows that $f(x)>f(y)$ holds true in every case. So $f$ is strongly-Lyapunov.
\end{proof}

Let us discuss an example where the previous result fail: a map that is Lyapunov but not strongly-Lyapunov in a non-compact space. 

\begin{example}\label{ex-L-no-sL}	
	We take the manifold $N=[0,3]\times]0,+\infty[$ equipped with the coordinate system $(x,y)$ and the metric $ydx^2+dy^2$. Let $\varphi$ be the flow generated by the vector field $u(x,y)=(\sin(\pi x),-\pi y\cos(\pi x))$. 
	For $i$ in $\intint{0,3}$, the curve $\delta_i=\{i\}\times\RR$ is an orbit of $\varphi$.
	Every other orbit is asymptotic to two~$\delta_i$. Every pseudo-orbit remains close to one of the $\delta_i$ after some time, so it escapes every compact in finite time. Thus, the recurrence set of $\varphi$ is empty.
	
	We build a Lyapunov function as follows. Take a smooth non-decreasing map $g\colon[0,3]\to[-2,2]$ that satisfies $g([0,1])=-2$ and $g([2,3])=2$. Let $f\colon N\to\RR$ be the map defined by:
	$$f(x,y)=\cos(\pi x)y+g(x).$$
	The derivative of $f$ in the direction of the flow is
	$$df\circ u(x,y)=-\pi y + g'(x)\sin(\pi x)<0$$
	which implies that $f$ is decreasing along the flow, so it is Lyapunov.
	
	We consider the two points $p=(0,1)$ and $q=(3,1)$ in $N$, which satisfy $f(p)=-1<1=f(q)$.
	For any $\epsilon>0$ and $t>\tfrac{\ln(3)-\ln(\epsilon)}{\pi}$, the concatenation of $\varphi_{[0,t]}(p)$ with $\varphi_{[-t,0]}(q)$ yields an~$\epsilon$-pseudo-orbit from $p$ to $q$. It has one jump from $\varphi_t(p)=(1,\e^{-\pi t})$ to $\varphi_{-t}(q)=(3,\e^{-\pi t})$, which are at distance at most $3\e^{-\pi t}<\epsilon$. It follows that $p\recto q$ holds. Together with $f(p)<f(q)$, it implies that $f$ is not strongly-Lyapunov.
\end{example}

Another direct consequence of Lemma~\ref{lem-blocking-box} is the following.

\begin{lemma}\label{lem-po-nincrease-preL}
	Assume that $N$ is compact. Let $f\colon N\to\RR$ be a continuous function, that is constant on each recurrence chain, non-increasing along the flow, and so that $f(\Rec)$ has at most countably many values.
	Then $f$ is pre-Lyapunov.
\end{lemma}


\subsection{Asymptotic pseudo-directions}\label{sec-ass-ps-dir}

Take an~$\epsilon$-pseudo-orbit $\gamma$ and $\Delta$ its discontinuity set. We call \emph{$\epsilon$-realization} of $\gamma$ any curve obtained by cutting $\gamma$ along $\Delta$, and for any $u$ in $\Delta$, inserting in the place of $u$ a continuous curve from $\gamma(u^-)$ to $\gamma(u)$ that remains in the~$\epsilon$-ball centered at $\gamma(u)$. 
When~$\epsilon$ is small enough, every ball of radius~$\epsilon$ is contractible. So any two curves that have the same ends, and that remain in a common~$\epsilon$-ball, are homotopic.
Then the homology class of any~$\epsilon$ realization of $\gamma$ (relatively to its endpoints) depends only on $\gamma$ itself.
For convenience, we say that~$\epsilon$ is \emph{small} when every~$\epsilon$-ball is contractible.

Given $\epsilon>0$ small and a periodic~$\epsilon$-pseudo-orbit~$\gamma$, we denote by $[\gamma]$ the homology class of any~$\epsilon$-realization of $\gamma$. 
Denote by $\len(\gamma)$ is the length of~$\gamma$.
We define the set
$$D_{\varphi,\epsilon,T}=\clos\left(\left\{\tfrac{1}{\len(\gamma)}[\gamma]\in H_1(M,\RR), \gamma\text{ a periodic $(\epsilon,T)$-pseudo-orbit}\right\}\right)$$ 
to be the closure of the set of elements $\tfrac{1}{\len(\gamma)}[\gamma]$. Note that $\tfrac{1}{\len(\gamma)}[\gamma]$ remains in a bounded region of $H_1(M,\RR)$, so $D_{\varphi,\epsilon,T}$ is compact.

\begin{definition}
	We define $D_\varphi=\bigcap_{\epsilon>0}D_{\varphi,\epsilon,T}$, which we call the \emph{set of asymptotic pseudo-directions} of $\varphi$.
\end{definition}

Notice that $D_\varphi$ does not depend on $T$, nor on the choice of metric on~$M$. The definition differ from the one given by Fried \cite{Fried82}. We prove in Appendix~\ref{app-asymptotic-cone} that Fried's set of asymptotic directions and $D_\varphi$ span the same convex set. Our main results depend on the convex hull of $D_\varphi$, so we may switch $D_\varphi$ and Fried's set.

\section{Reducing pseudo-orbits}\label{sec-reduction-pso}

In order to characterize quasi-Lyapunov classes, especially in the view of the real coefficient case, we need to a technical tool. 
We describe how to reduction a pseudo-orbit, of any length, into finitely many pseudo-orbits of bounded length. Recall that when~$\epsilon$ is small enough, any two~$\epsilon$-realizations of a given~$\epsilon$-pseudo-orbit $\gamma$ are homologous, and $[\gamma]$ denote their common homology class. For shortness, we say that~$\epsilon$ is small when it satisfies this condition.

\begin{lemma}\label{lem-cut-pso-2}
	For any small $\epsilon>0$, there exist $\epsilon'>0$ and $L>0$ which satisfy the following. For any periodic $\epsilon'$-pseudo-orbit $\gamma$, there exist finitely many periodic~$\epsilon$-pseudo-orbits $\delta_1\cdots\delta_r$ which satisfy:
	\begin{itemize}
		\item $\gamma$ is homologous to $\sum_i\delta_i$ inside $H_1(M,\ZZ)$,
		\item we have $\len(\delta_i)\leq L$ for all $i$, and $\len(\gamma)=\sum_i\len(\delta_i)$.
	\end{itemize}
\end{lemma}

The lemma admits a non-periodic counter-part.

\begin{lemma}\label{lem-cut-pso-1}
	For any small $\epsilon>0$, there exist $\epsilon'>0$ and $L>0$ which satisfy the following. For any $\epsilon'$-pseudo-orbit $\gamma$, there exist finitely many~$\epsilon$-pseudo-orbits $\delta_0,\delta_1\cdots\delta_r$ which satisfy:
	\begin{itemize}
		\item for all $i>0$, $\delta_i$ is periodic,
		\item $\delta_0$ starts and ends at the same points as $\gamma$,
		\item $\gamma$ is homologous to $\sum_i\delta_i$ relatively to $\partial\gamma$,
		\item we have $\len(\delta_i)\leq L$ for all $i$, and $\len(\gamma)=\sum_i\len(\delta_i)$.
	\end{itemize}
\end{lemma}

\begin{proof}[Proof of Lemma~\ref{lem-cut-pso-2}]
	Take $\epsilon>\epsilon'>0$. Up to taking them smaller, we may assume that every ball of radius $\epsilon+\epsilon'$ is contractible.
	Cover~$M$ with finitely many open balls of radius $\tfrac{\epsilon'}{2}$, and denote them by $B_1\cdots B_p$. We set $L=2pT$.

	Let $\gamma$ be an~$\epsilon$-pseudo-orbit, $l=\len(\gamma)$ its length, which we assume larger than $L$. Since the discontinuity times of $\gamma$ are distant by at least $T$, we may choose a cyclically increasing sequence $t_1\cdots t_n$ in $\bfrac{\RR}{l\ZZ}$ that satisfies the following:
	\begin{itemize}
		\item every discontinuity time $\gamma$ is one of the $t_i$,
		\item $t_{i+1}$ belongs to $[t_i+T,t_i+2T]$.
	\end{itemize}
	For all index $i$, choose some $\sigma(i)$ for which $t_i$ belongs to $B_{\sigma(i)}$. Denote by $I_k$ the interval $[t_k,t_{k+1}[$, seen as a subset of $\bfrac{\RR}{l\ZZ}$. 
	We build an oriented graph $G$ whose set of vertices is $\intint{1,p}$, with an oriented edge $e_k$ from $\sigma(k)$ to $\sigma(k+1)$ for all $k$ in $\intint{1,n}$. The sequence $e_1\cdots e_n$ induces an Eulerian oriented cycle inside $G$, so $G$ is Eulerian. It follows that there exists finitely many cycles $c_1\cdots c_r$ inside $G$, which satisfies the following:
	\begin{itemize}
		\item each cycle $c_k$ goes through any vertex at most once, thus $\len(c_k)\leq p$ holds,
		\item each edge of $G$ is travel exactly once by the union of the $c_k$.
	\end{itemize} 
	Given a cycle $c_k$, we define $\gamma_k$ to be the periodic $\epsilon'$-pseudo-orbit obtained as follows. Let $e_{i_1}\cdots e_{i_r}$ be the sequence of edges in $c_k$, and define $\gamma_i$ to be the concatenation, in cyclic order, of the orbit arcs $\gamma_{|I_{i_1}}\cdots\gamma_{|I_{i_r}}$. It is indeed a~$\epsilon$-pseudo-orbit since every $\gamma_{|I_k}$ is an orbit arc of length at least $T$, and every jump are by a distance no more than the diameter of the $B_k$.

	The length of $\gamma_i$ is at most $2T\len(c_k)\leq 2pT$, and the sum of their lengths is equal to $l=\len(\gamma)$. 
	It remains to prove that $\gamma$ is homologous to $\sum_i\gamma_i$. 
	
	Recall that the homology class of $\gamma$ is equal to the homology class of any~$\epsilon$-realization of $\gamma$. And an~$\epsilon$-realization is obtained by inserting~$\epsilon$-small arcs at the discontinuity time of $\gamma$, to make it continuous. Call \emph{connecting arcs} the added curves for an arbitrary choice of~$\epsilon$-realization, and similarly for $\gamma_i$.
	The difference $[\gamma]-\sum_i[\gamma_i]$ is homologous to the sum of connecting arcs in $\gamma$ minus the connecting arcs in the $\gamma_i$. 

	Take a vertex $v$ in $\intint{1,p}$ of $G$ and $k$ with $\sigma(k)=v$. Denote by $\delta_k$ the connecting arc in $\gamma$ that lies just after $\gamma(I_{k-1})$ and just before $\gamma(I_k)$. There is a unique cycle $c_l$ that contains $e_{k-1}$, and there is an edge $e_{\tau(k)}$ that comes just after $e_{k-1}$ in $c_l$. The corresponding pseudo-orbit $\gamma_l$ admits a connecting arc $\delta'_k$ that lies just after $\gamma(I_{k-1})$ and just before $\gamma(I_{\tau(k)})$. Note that $k\mapsto\tau(k)$ yields a permutation of $\sigma^{-1}(\{v\})$.	
	We claim that the sum $\sum_k(\delta_k-\delta'_k)$, for $k$ with $\sigma(k)=v$, is homologous to zero. Indeed, the boundary of the sum is zero, and every arc $\delta_k,\delta'_k$ remains in a common ball or radius $\epsilon+\epsilon'$, which is contractible. Thus, the sum is null homologous in that ball and in~$M$ too. 	
	
	Every connecting arcs in $[\gamma]-\sum_i[\gamma_i]$ is represented exactly once in the above discussion, for one vertex $v$ of $G$. So adding the sums $\sum(\delta_k-\delta'_k)$ over all values of $v$ yields that $[\gamma]-\sum_i[\gamma_i]$ is null homologous.
\end{proof}

\begin{proof}[Proof of Lemma~\ref{lem-cut-pso-1}]
	The proof of the non-periodic case is very similar to the periodic one. One needs to replace $\bfrac{\RR}{\len(\gamma)\ZZ}$ with $[0,\len(\gamma)[$, adapt the Eulerian graph part and the discussion about the connecting arcs. We let the details to the reader.
\end{proof}


The discussion above has an interesting finitness consequence.

\begin{lemma}\label{lem-cut-pso-finite}
	For all small $\epsilon>0$ and all $c>1$, there exist some $\epsilon'>0$ and a finite set of periodic~$\epsilon$-pseudo-orbits $\delta_1\cdots\delta_n$ which satisfy the following. For any periodic $\epsilon'$-pseudo-orbits $\gamma$, there exists $a_1\cdots a_n$ in $\NN$ so that:
	\begin{itemize}
		\item $\gamma$ is homologous to the sum $\sum_ia_i\delta_i$,
		\item $c^{-1}\sum_ia_i\len(\delta_i)<\len(\gamma)<c\sum_ia_i\len(\delta_i)$.
	\end{itemize}
\end{lemma}

It has the immediate corollary:

\begin{corollary}\label{cor-pso-decomposition}
	For any small $\epsilon>0$, there exist an $\epsilon'$ in $]0,\epsilon[$ and a subset~$\Delta$ of $H_1(M,\RR)$ that satisfy:
	$$D_\varphi\subset D_{\varphi,\epsilon',T}\subset[1-\epsilon,1+\epsilon]\cdot\conv(\Delta)\subset[1-\epsilon,1+\epsilon]\cdot\conv(D_{\varphi,\epsilon,T}),$$
	and so that $\Delta$ is made of the homology classes of finitely many periodic~$\epsilon$-pseudo-orbits.
\end{corollary}

\begin{proof}
	Apply Lemma~\ref{lem-cut-pso-2} to $\tfrac{\epsilon}{2}$, and denote by $\epsilon'>0$ and $L>0$ the constants in the lemma. Denote by $E$ the set of $\tfrac{\epsilon}{2}$-pseudo-orbits of length at most $L$, and $\wb E$ its closure, inside the set of~$\epsilon$-pseudo-orbits of length at most $L$. Then $\wb E$ is compact. So extract finitely many~$\epsilon$-pseudo-orbits $\delta_1\cdots\delta_n$ in $\wb E$ that satisfies that any $\tfrac{\epsilon}{2}$-pseudo-orbit $\gamma$ of length at most $L$ is homologous to some $\delta_i$ with $c^{-1}\len(\delta_i)<\len(\gamma)<c\len(\delta_i)$.

	Given an $\epsilon'$-pseudo-orbits $\gamma$, apply Lemma~\ref{lem-cut-pso-2} to $\gamma$. It yields finitely many periodic $\tfrac{\epsilon}{2}$-pseudo-orbits $\gamma_i$ of length at most $L$, and so that $\gamma$ is homologous to $\sum_i\gamma_i$. Each $\gamma_i$ is homologous to some $\delta_{\sigma(i)}$ with 
	$$c^{-1}\len(\delta_{\sigma(i)})<\len(\gamma_i)<c\len(\delta_{\sigma(i)}).$$ 
	Then $\gamma$ is homologous to the sum of the $\delta_{\sigma(i)}$ and 
	$$c^{-1}\sum_i\len(\delta_{\sigma(i)})<\len(\gamma)<c\sum_i\len(\delta_{\sigma(i)})$$
	holds true.
\end{proof}

\section{Quasi-Lyapunov classes}\label{sec-spectral-decomposition}

Let $\alpha$ be in $H^1(M,\RR)$ and $f\colon\wh M_\alpha\to\RR$ be an $\alpha$-equivariant and continuous map. Given $C\geq 0$, we say that $f$ is \emph{$C$-quasi-Lyapunov} if there exists $\epsilon>0$ which satisfies the following. For any~$\epsilon$-pseudo-orbit $\gamma\colon[0,\len(\gamma)[\to\wh M_\alpha$ and any $t$ in $[0,\len(\gamma)[$, we have $f\circ\gamma(0)\geq f\circ\gamma(t)-C$. We say that $f$ is \emph{quasi-Lyapunov} if it is $C$-quasi-Lyapunov for some~$C$. 

Given two $\alpha$-equivariant continuous map $f,g\colon\wh M\to\RR$, $f-g$ is invariant under the action by $H_1(M,\ZZ)$, so it is bounded. Thus $f$ is quasi-Lyapunov if and only $g$ is quasi-Lyapunov.
A cohomology class $\alpha$ in $H^1(M,\RR)$ is said \emph{quasi-Lyapunov} if any $\alpha$-equivariant and continuous map $f\colon\wh M\to\RR$ is quasi-Lyapunov. Note that $\alpha=0$ is always quasi-Lyapunov. For shortness reasons, we gave a different definition in the introduction. Lemma~\ref{lem-quasi-L-char}, stated further down, states that the two definitions are equivalent.

Consider the following property on $f$: we have $f(x)\geq f(y)-C$ whenever $x\recto y$ is satisfied (without~$\epsilon$). It is implied by being $C$-quasi-Lyapunov, and even equivalent when $\alpha$ has integer coefficients. But in the general case, the second property is too weak for our applications.

In comparaision with Fried's work \cite{Fried82} on global section, we have the following result. The sign difference is due to a convention difference.

\begin{lemma}\label{lem-quasiL-hom-inequality}
	If a class $\alpha$ in $H^1(M,\RR)$ is quasi-Lyapunov, then $\alpha(D_\varphi)\leq 0$ holds.
\end{lemma}

This is a direct consequence of Lemma~\ref{lem-quasi-L-char} stated further down. 
For hyperbolic enough flow, the converse to Lemma~\ref{lem-quasiL-hom-inequality} should be true, but it is not true in general. 
Let us give two examples of 1-forms $\alpha$ which satisfy $\alpha(D_\varphi)\leq 0$ but which are not quasi-Lyapunov.

\begin{example}\label{ex-linear-flow}
	Take $n\geq 2$. Let $\varphi$ be a linear flow on $M=\bfrac{\RR^n}{\ZZ^n}$ directed by some non-zero vector $X$ in $\RR^n$. Let $\alpha$ in $H^1(M,\RR)$ be a class that contains $X$ in its kernel. Then we have $D_\varphi=\{X\}$, so $\alpha(D_\varphi)=0$ holds. But $\alpha$ is not quasi-Lyapunov, as one can build a periodic~$\epsilon$-pseudo-orbit that only jumps in the positive direction of $\alpha$. 
\end{example}

\begin{example}\label{ex-not-quasiL-1} 
	Let $\varphi$ be a flow on $\bfrac{\RR}{\ZZ}$ that has one fixed points at $0$ and which is increasing outside $0$. Let $\alpha$ be the cohomology class of $dt$. 

	For any small $\epsilon>0$ and any $n>0$, there is a periodic $\epsilon$-pseudo-orbits $\gamma$ that turns $n$ times around~$\bfrac{\RR}{\ZZ}$, in the positive direction. Thus, $\alpha([\gamma])=n>0$ holds. So $\alpha$ is not quasi-Lyapunov. 
	
	Take a periodic $\epsilon$-pseudo-orbit $\gamma$ that turns $n$ times around $\bfrac{\RR}{\ZZ}$. It jumps $n$ times across 0. So it spends a time at least $n\times t(\epsilon)$ on the interval $]0,\tfrac{1}{2}]$, with $t(\epsilon)\xrightarrow[\epsilon\to 0]{}+\infty$. So $\tfrac{1}{\len(\gamma)}[\gamma]$ goes to zero when $\epsilon$ goes to zero. It follows that $D_\varphi$ is reduced to $\{0\}$, so $\alpha(D_\varphi)=0$ holds.
\end{example}

Not all counterexamples satisfy $\alpha(D_\varphi)=0$; this is an artifact of taking the simplest counterexamples.

\subsection{Characterizations of quasi-Lyapunov classes}\label{sec-quasiL}

We give several characterizations of the quasi-Lyapunov property. Under a regularity assumption and with integer coefficients, they can be gathered into the following:

\begin{theorem}\label{thm-quasiL-char-general}
	Assume~$M$ smooth and $\varphi$ generated by a $\Class^1$ vector field. For any continuous flow $\psi$ close enough to $\varphi$, Given $\alpha$ in $H^1(M,\ZZ)\setminus\{0\}$, the following statement are equivalent:
	\begin{enumerate}[i.]
		\item $\alpha$ is quasi-Lyapunov,
		\item there exists a Lyapunov and $\alpha$-equivariant map $f\colon\wh M_\alpha\to\RR$,
		\item we have $-\infty\not\recto+\infty$ in $\wb M_\alpha$
		\item for any $T>0$, there exists $\epsilon>0$ for which $\alpha(D_{\varphi,\alpha,\epsilon,T})\leq 0$ holds, 
		\item we have $\alpha(D_\psi)\leq 0$.
	\end{enumerate}
\end{theorem}

The equivalence with $ii$ is the subject of Section~\ref{sec-quasi-L-criterion}. The other equivalences are consequences of the next results, each stated with less hypotheses. 


\begin{lemma}\label{lem-quasiL-infty-char}
	A class $\alpha$ in $H^1(M,\ZZ)\setminus\{0\}$ is quasi-Lyapunov if and only if we have $-\infty\not\recto+\infty$ for the flow on $\wb M_\alpha$, if and only we have $x\not\recto+\infty$ for all~$x$ in $\wh M_\alpha$.
\end{lemma}

Recall that the metric on $\wb M_\alpha$ does not restrict to the one on $\wh M_\alpha$. But for any compact $K$ subsets on $\wh M_\alpha$, the two metrics are comparable at small scale on $K$. So essentially, pseudo-orbits for one metrics are pseudo-orbits for the other inside compact subsets of $\wh M_\alpha$.

\begin{proof}
	It is clear that if $x\not\recto+\infty$ holds for some~$x$, then $\alpha$ is not quasi-Lyapunov. 

	Assume that $-\infty\recto +\infty$ holds. Take a compact subset $K\subset\wh M_\alpha$ so that the projection $K\xrightarrow{\pi_\alpha}M$ is surjective. Any sequence of~$\epsilon$-pseudo-orbits from $-\infty$ to $+\infty$, with~$\epsilon$ that goes to zero, accumulates on a point~$x$ in $K$ that satisfies $x\recto+\infty$.

	Assume now $\alpha$ not quasi-Lyapunov. Take an $\alpha$-equivariant and continuous map $f\colon\wh M_\alpha\to\RR$. For any $C\geq0$ and $\epsilon>0$, there exists a~$\epsilon$-pseudo-orbit that goes from a point~$x$ to a point $y$, with $f(y)\geq f(x)+3C$. Up to assuming $C$ large enough, and replacing~$x$ and $y$ by $n\cdot x$ and $n\cdot y$ for some~$n$ in $\ZZ$, we may assume that $f(y)\geq C$ and $f(x)\leq C$ hold. Taking $C$ large enough, there exists a~$\epsilon$-pseudo-orbit from $-\infty$ to $+\infty$ that jumps from $-\infty$ to~$x$, follows the previous~$\epsilon$-pseudo-orbit, and jump back to $+\infty$. It follows that $-\infty\recto +\infty$ holds.
\end{proof}

\begin{lemma}\label{lem-quasi-L-char}
	A class $\alpha$ in $H^1(M,\RR)$ is quasi-Lyapunov, if and only if we have $\alpha(D_{\varphi,\epsilon,T})\leq 0$ for some/any $\epsilon>0$ small enough.
\end{lemma}

\begin{proof}
	The case $\alpha=0$ is clear, so we may assume $\alpha\neq 0$.
	Let $f\colon\wh M\to\RR$ be an $\alpha$-equivariant and continuous map.

	Assume first that $\alpha$ is quasi-Lyapunov. Take $\epsilon>0$ and $C\geq 0$ so that $f$ is $C$-quasi-Lyapunov. Let $\gamma$ be a periodic~$\epsilon$-pseudo-orbit in~$M$. In $\wh M$, $\gamma$ can be lifted into an~$\epsilon$-pseudo-orbit $\gamma'$ that goes from a point~$x$ in $\wh M$ to $[\gamma]\cdot x$. Concatenating the pseudo-orbits $k[\gamma]\cdot\gamma'$, for $0\leq k<n$ in increasing order, yields an~$\epsilon$-pseudo-orbit from~$x$ to $n[\gamma]\cdot x$. Since $f$ is $\alpha$-equivariant, we have $f(n[\gamma]\cdot x)=f(x)+n\alpha([\gamma])$. And since it is bounded above by $f(x)+C$ for all~$n$, we have $\alpha([\gamma])\leq 0$. The inequality $\alpha(D_{\varphi,\epsilon,T})\leq 0$ follows.

	We prove the converse. Assume that $\alpha(D_{\varphi,\epsilon,T})\leq 0$ holds for some $\epsilon>0$. Take $\epsilon'>0$ given by Lemma~\ref{lem-cut-pso-1}.
	Let $\gamma$ be an $\epsilon'$-pseudo-orbit in $\wh M$ from a point~$x$ to a point $y$. Then the image $\pi_\alpha(\gamma)$ of $\gamma$ in~$M$ is homologous, relatively to its endpoints, to an~$\epsilon$-pseudo-orbit $\delta_0$ of bounded length plus finitely many periodic~$\epsilon$-pseudo-orbit $\delta_1\cdots \delta_r$. The pseudo-orbit $\delta_0$ lifts inside~$\wh M$ to a~$\epsilon$-pseudo-orbit $\wh\delta_0$, that goes from~$x$ to a point $y'$.	

	We estimate $f(x)-f(y')$ and $f(y')-f(y)$ separately.
	The boundary of the 1-chain $\gamma-\wh\delta_0$ is equal to $y-y'$. So $y$ is the image of $y'$ under the action of the homology class:
	$$[\pi_\alpha(\gamma)]-[\delta_0]=\sum_{k\geq 1}[\delta_k].$$
	Thus, we have $f(y)=f(y')+\sum_{k\geq 1}\alpha([\delta_k])$.
	The homology class of $\delta_i$ belongs to $D_{\varphi,\epsilon,T}$ so we have $\alpha([\delta_k])\leq 0$ holds true. Thus, $f(y')\geq f(y)$ holds.
	Since the length of $\wh\delta_0$ is bounded, the difference $|f(x)-f(y')|$ is bounded by a constant $C\geq0$, independent on $\gamma$. It follows $f(x)\geq f(y)-C$ holds. Hence,~$f$ is $C$-quasi-Lyapunov, and $\alpha$ is quasi-Lyapunov. 
\end{proof}

It has another consequence.

\begin{corollary}
	Given $\alpha$ in $H^1(M,\RR)$, being quasi-Lyapunov for $\varphi$ is an open property on~$\varphi$, for the $\Class^0$ topology.
\end{corollary}

Given two map $\gamma,\delta\colon\bfrac{\RR}{l\ZZ}\to M$, $\gamma$ is said to \emph{$\eta$-shadows} $\delta$ if for all $t$, $\gamma(t)$ and $\delta(t)$ are at distance less than $\eta$.

\begin{proof}
	We prove that the set of flows $\psi$ flows for which $\alpha$ is not quasi-Lyapunov is closed. Let $\psi_n$ be a sequence of flows that converges toward~$\varphi$, $\epsilon_n>0$ be a sequence that converges toward zero, and for every~$n$, $\gamma_n$ be a periodic $T$-pseudo-orbit with $\alpha([\gamma_n])>0$. 
	
	For any $\epsilon>0$, and any~$n$ large enough, $\gamma_n$ can be~$\epsilon$-shadowed by an $\epsilon$-pseudo-orbit of $\varphi$. When~$\epsilon$ is small enough, $\gamma_n$ and $\delta$ are homologous, so $\alpha([\delta])>0$ holds. It follows that $\alpha$ is not quasi-Lyapunov for $\varphi$.
\end{proof}

We give a last characterization, that is useful when considering generic properties of flows. 

\begin{proposition}
	Assume~$M$ smooth, and that $\varphi$ is generated by a $\Class^1$ vector field. A class $\alpha$ in $H^1(M,\RR)$ is quasi-Lyapunov for $\varphi$ if and only if for any continuous flow~$\psi$ on a $\Class^0$-neighborhood of $\varphi$, $\alpha(D_\psi)\leq 0$ holds.
\end{proposition}

The regularity assumptions are technical convenience and probably not optimal.

\begin{proof}
	We first assume that there exists a sequence of flows $\psi_n$ that converges toward~$\alpha$, and that satisfy $\alpha(D_{\psi_n})\not\leq 0$. For every~$n$, take a periodic $\tfrac{1}{n}$-pseudo-orbit $\gamma_n$ of $\psi_n$ that satisfies $\psi_n(\gamma_n)>0$. Fix some small $\epsilon>0$. For any~$n$ large enough, $\gamma_n$ can be~$\epsilon$-shadowed by a~$\epsilon$-pseudo-orbit $\delta$ of $\varphi$. When~$\epsilon$ is small enough, it follows that $\delta$ and~$\gamma_n$ are homologous, so $\alpha(\delta)>0$ holds. Therefore, $\alpha$ is not quasi-Lyapunov.
	
	We now assume that $\alpha$ is not quasi-Lyapunov. For convenience, we may assume that the metric on~$M$ is a smooth Riemannian metric.
	
	We shadow a pseudo-orbits with a variation of that notion. Denote by~$X$ the vector field that generates $\varphi$. Let $U$ be a small neighborhood of the image of $X$ in $TM$. A curve $\gamma\colon\bfrac{\RR}{l\ZZ}\to M$ is called a periodic \emph{$U$-pseudo-orbit} if it is continuous, piecewise $\Class^1$, and its left and right derivatives lie in~$U$. 

	We claim that for any $\eta>0$ there exists some $\epsilon>0$ for which any periodic~$\epsilon$-pseudo-orbit $\gamma$ can be $\eta$-shadowed by a periodic $U$-pseudo-orbit~$\delta$. For that, take some $t$ in $[T,2T]$, some points $x,y$ in~$M$ at distance less than~$\epsilon$, and $z=\varphi_{-t}(x)$. Denote by $c\colon[0,\epsilon]\to M$ a geodesic, at constant speed at most one, from~$x$ to $y$. Then the map $\delta\colon[0,t]\to M$ by:
	$$\delta(s)=\varphi_{s-t}\circ c\left(\frac{s\epsilon}{t}\right)$$
	$\eta$-shadows the orbit arc $\varphi_{[0,t]}(z)$ when~$\epsilon$ is small enough.	
	Note that $\delta'$ converges uniformly toward $X$ when~$\epsilon$ goes to zero, so it lies in $U$ if~$\epsilon$ is small enough. To prove the claim, we cut $\gamma$ in continuity intervals whose length are in $[T,2T]$, and we replace them with curves constructed similarly to~$\delta$.

	Take a small neighborhood $U$ of $X$ as above. From Lemma~\ref{lem-quasi-L-char} and from the claim above, there exists a periodic $U$-pseudo-orbit $\delta\colon\bfrac{\RR}{l\ZZ}\to M$ satisfies $\alpha(\delta)>0$. Up to changing $\delta$ close to the fixed points of $\varphi$, we may assume that $\delta$ is locally injective.
	We cut and past $\delta$ to make it injective. Take two distinct times $u,v$ in $\bfrac{\RR}{l\ZZ}$ which satisfy $\gamma(u)=\gamma(v)$. Denote by $I,J$ the two closed intervals separated by $u$ and $v$, and by $I^\circ$ and $J^\circ$ the circle obtained by identifying the end points in $I$ and $J$. Then $\delta$ restricts to two periodic $U$-pseudo-orbits $\gamma_{|I^\circ}$ and $\gamma_{|J^\circ}$, whose sum is homologous to $\delta$. Thus, one of them, let say the first one, satisfies $\alpha(\gamma_{|I^\circ})>0$. 
	Up to applying the previous procedure finitely many times, we may assume that $\delta$ is injective.

	We build a flow $\psi$ close to $\varphi$, for which $\delta$ is a periodic orbit. Since $\delta$ is injective and piecewise $\Class^1$, we can take a smooth tubular neighborhood $V$ of~$\delta$. And similarly there exists a continuous vector field $Y$ on $V$, piecewise~$\Class^1$, and that coincide with the derivatives of $\delta$. Up to making $V$ smaller, we may also assume that~$Y$ lies point-wise in $U$. Take a $\Class^1$ function $h\colon V\to [0,1]$, with $h\equiv 1$ on $\delta$, and~$h\equiv 0$ on a neighborhood of $\partial V$. Consider the vector field $Z$ on~$M$, given by $Z(p)=h(p)Y(p)+(1-h(p))X(p)$ for any $p$ in $V$, and $Z(p)=X(p)$ for $p$ outside $V$. It is piecewise of class $\Class^1$, so it is Lipschitz. 
	
	Denote by $\psi$ the generated flow. If $U$ is taken convex and small enough, then $Z$ lies in $U$ and is very close to $X$, Thus $\psi$ can be taken arbitrarily close to $\varphi$. By definition, $\delta$ is a periodic orbit of $\psi$, so $\alpha(D_\psi)\not\leq0$ holds. The conclusion follows.
\end{proof}

\subsection{The $\alpha$-recurrent set}

Fix a class $\alpha$ in $H^1(M,\RR)$, not necessarily quasi-Lyapunov. A point~$x$ in~$M$ is said \emph{$\alpha$-recurrent} if for every $\epsilon>0$ small and every $T>0$, there exist a periodic~$\epsilon$-pseudo-orbit $\gamma$ that passes through~$x$ and that satisfy $\alpha([\gamma])=0$. We denote by $\Rec_\alpha$ the set of $\alpha$-recurrent point, called \emph{$\alpha$-recurrent set}. Two $\alpha$-reccurent points $x,y$ in~$M$ are said to be $\alpha$-equivalent if for all $\epsilon,T>0$ there exists a periodic~$\epsilon$-pseudo-orbit $\gamma$ that passes through~$x$ and $y$ and with $\alpha([\gamma])=0$. The $\alpha$-equivalence relation is an equivalence relation. We call \emph{$\alpha$-recurrence chains} the $\alpha$-equivalence classes. Note that the $\alpha$-equivariance chains are invariant by the flow.

Recall the covering map notation $\wh M_\alpha\xrightarrow{\pi_\alpha} M$.
Given a recurrence chain~$R$ in $\wh M_\alpha$, $\pi_\alpha(R)$ is an $\alpha$-recurrence chain in~$M$. Conversely, any $\alpha$-recurrence chain can be obtained that way.

We continue with some elementary properties on the $\alpha$-recurrent set.

\begin{lemma}
	The $\alpha$-recurrence set and the $\alpha$-recurrence chains are closed inside~$M$.
\end{lemma}

\begin{proof}
	Let $(x_n)_n$ be a sequence of $\alpha$-recurrent points that converges toward some point $y$ in~$M$. 
	Fix some small $\epsilon>0$. Take~$n$ very large, $\eta>0$ very small, and $\gamma$ a periodic $\eta$-pseudo-orbit with $\gamma(0)=x_n$ and with $\alpha([\gamma])=0$. We may assume that the length of $\gamma$ is at least $2T$. Take $t_0$ in $[T,2T]$ to be either equal to the first discontinuity time of $\gamma$ if it is smaller than $2T$, or $t_0=T$ if $\gamma$ is continuous on $[0,2T]$. Define a pseudo-orbit $\delta$ of the same length as $\gamma$, that satisfies $\delta(t)=\varphi_t(x)$ for all $t$ in $[0,t_0[$, and that coincide with $\gamma$ outside $[0,t_0[$. If~$n$ is large enough and $\eta$ is small enough, $\delta$ is a periodic~$\epsilon$-pseudo-orbit that satisfies $\alpha([\delta])=0$. So~$x$ is $\alpha$-recurrent.
	
	The same argument, assuming all $x_n$ to belong to the same $\alpha$-recurrence chain, shows that any $\alpha$-recurrence chain is closed.
\end{proof}

\begin{lemma}\label{lem-rec-con-comp}
	The $\alpha$-recurrence chains are the connected components of the $\alpha$-recurrent set $\Rec_\alpha$.
\end{lemma}

Denote by $V_\epsilon(X)$ the~$\epsilon$-neighborhood of a subset $X\subset M$.

\begin{proof}
	Let~$R$ be an $\alpha$-recurrence chain, assumed to be not connected. It is the union of two disjoint open subsets $A,B$. Then $A$ and $B$ are closed in~$R$, and even compact. Thus, they are separated by a positive distance $3\epsilon$. For~$n$ in $\NN$, let $\gamma_n$ be a periodic $\tfrac{1}{n}$-pseudo-orbits that passes through a point~$x$ in $A$. Any accumulation point of the sequence $\gamma_n$, when~$n$ goes to $+\infty$, is $\alpha$-equivariant to~$x$, so it belongs to~$R$. Thus, the $\gamma_n$ do not accumulate outside $V_\epsilon(A)\cup V_\epsilon(B)$. Since $V_\epsilon(A)$ and $V_\epsilon(B)$ are disjoint, it follows that~$\gamma_n$ remains inside $V_\epsilon(A)$ when~$n$ is large enough. It contradicts that $A$ and $B$ are $\alpha$-equivariant. Thus, any $\alpha$-recurrence chain is connected.

	Let $A$ be a connected component of $\Rec_\alpha$ and $x_1,x_2$ be two points in $A$. Fix some small $\epsilon>0$. Since $A$ is connected and since the~$\epsilon$-balls are path connected, the set $V_\epsilon(A)$ is path connected. It follows that there exists a sequence $y_k$ in $V_\epsilon(A)$, for $1\leq k\leq p$, which satisfies $y_1=x$ and $y_p=y$ and that the distance between $y_k$ and $y_{k+1}$ is not more than~$\epsilon$.
	Denote by $z_k$ a point in $A$ at distance at most~$\epsilon$ from $y_k$, and by $\gamma_k$ a periodic~$\epsilon$-pseudo-orbit that start at $z_k$, and with $\alpha([\gamma_k])=0$. We may choose $z_1=x_1$ and $z_p=x_2$. It follows from above than the distance between $z_k$ and $z_{k+1}$ is at most $3\epsilon$.
	
	We describe a periodic $3\epsilon$-pseudo-orbit $\delta$ that goes through $x_1$ and~$x_2$. We view $\gamma_k$ as a (non-periodic)~$\epsilon$-pseudo-orbit that starts and ends at $z_k$. The concatenation of the $\gamma_k$, in increasing order, is a $3\epsilon$-pseudo-orbit from~$x_1$ to $x_2$. Similarly, the concatenation of the $\gamma_k$, in decreasing order, is a $3\epsilon$-pseudo-orbit from $x_2$ to $x_1$. Denote by $\gamma$ the concatenation of these two. It is a periodic $3\epsilon$-pseudo-orbit that passes through $x_1$ and~$x_2$. It satisfies $[\gamma]=2\sum_k[\gamma_k]$, so $\alpha([\gamma])=0$ holds. Thus, $x_1$ and $x_2$ are $\alpha$-equivariant. Therefore, the connected components of $\Rec_\alpha$ are the $\alpha$-recurrence chains.	
\end{proof}

Lemma~\ref{lem-quasi-L-char} has the following consequence, which is a reinterpretation of Fried's results.

\begin{proposition}\label{prop-rec-empty}
	We have $\alpha(D_\varphi)<0$ if and only if $\alpha$ quasi-Lyapunov and $\Rec_\alpha$ is empty.
\end{proposition}

Note that $\alpha$ must be non-zero for the conditions to hold.

\begin{proof}
	Assume that $\alpha(D_\varphi)<0$ holds true. Since $D_\varphi$ is the decreasing intersection of the compact sets $D_{\varphi,\epsilon,T}$, we have $\alpha(D_{\varphi,\epsilon,T})<0$ for some $\epsilon>0$. From Lemma~\ref{lem-quasi-L-char}, $\alpha$ is quasi-Lyapunov. Additionally, there is no periodic~$\epsilon$-pseudo-orbit $\gamma$ with $\alpha([\gamma])=0$, so there is no $\alpha$-recurrent point.

	Assume $\alpha$ quasi-Lyapunov and $\Rec_\alpha$ empty. We assume that $\alpha(D_\varphi)\nless0$ holds. Then for all small $\epsilon>0$, there exists a periodic~$\epsilon$-pseudo-orbit $\gamma_\epsilon$ with $\alpha([\gamma_\epsilon])\geq0$. Since $\alpha$ is quasi-Lyapunov, we have $\alpha([\gamma_\epsilon])=0$ when~$\epsilon$ is small enough. Any accumulation point of $\gamma_\epsilon$ (when~$\epsilon$ goes to zero) is $\alpha$-recurrent. It contradicts that $\Rec_\alpha$ is empty, thus $\alpha(D_\varphi)<0$ holds.
\end{proof}

When $\alpha$ has integer coefficients, we additionally have.

\begin{proposition}\label{prop-alpha-rec-chain}
	Let $\alpha$ in $H^1(M,\ZZ)\setminus\{0\}$ be quasi-Lyapunov and~$R$ be a recurrence chain inside~$\wh M_\alpha$. Then~$R$ is compact and the map $R\xrightarrow{\pi_\alpha}\pi_\alpha(R)$ is a homeomorphism. 
\end{proposition}


\begin{proof}
	Let $f\colon M\to\bfrac{\RR}{\ZZ}$ be continuous function, cohomologous to $\alpha$, and let $\wh f\colon\wh M_\alpha\to\RR$ be an $\alpha$-equivariant lift of $f$ to $\wh M_\alpha$. Then $\wh f$ is $C$-quasi-Lyapunov, for some $C\geq 0$. Let $\epsilon>0$ be given by the $C$-quasi-Lyapunov property of $\wh f$.
	Let~$R$ be a recurrence chain in $\wh M_\alpha$. Then $\wh f(R)$ is an interval $I$ of length at most $C$. Since $\wh M_\alpha\to M$ is a $\ZZ$-covering, $\wh f^{-1}(I)$ is compact. Therefore,~$R$ is compact.

	Now, assume that the projection $R\xrightarrow{\pi_\alpha} M$ is not injective. There exists two points in~$R$ sent by $\pi_\alpha$ to the same point $p$ in~$M$. That is there exist~$x$ in~$R$ and~$n$ in $\ZZ\setminus\{0\}$ for which $n\cdot x$ also belongs to~$R$. Then~$R$ contains all the points $(kn)\cdot x$ for all $k$ in $\ZZ$, which contradicts the compactness of~$R$. Therefore, $R\xrightarrow{\pi_\alpha} M$ is injective and thus it is an embedding. Its image is an $\alpha$-recurrence chain by definition, and the map $R\xrightarrow{\pi_\alpha}\pi_\alpha(R)$ is a homeomorphism.
\end{proof}

We end with a relation between the recurrent sets in $\wh M_\alpha$ and $\wb M_\alpha$.

\begin{lemma}\label{lem-rec-identify}
	When $\alpha$ is quasi-Lyapunov, a subset of $\wh M_\alpha$ is a recurrence chain for the flow on $\wh M_\alpha$ if and only if it is a recurrence chain for the flow on $\wb M_\alpha$. 
\end{lemma}

So when $\alpha$ is quasi-Lyapunov, we can identify the recurrence chains in $\wh M_\alpha$ with the ones in $\wb M_\alpha$ that are not $\{\pm\infty\}$.

Note that the conclusion does not necessarily hold when $\alpha$ is not quasi-Lyapunov. For instance for some pair of points $x,y$ in $\wh M_\alpha$, a pseudo-orbit can start at $x$, go very close to $\pm\infty$, do a jump that is small in $\wb M_\alpha$ but large in $\wh M_\alpha$, and then go back to $y$.

\begin{proof}
	Let $f\colon\wh M_\alpha\to\RR$ be a $C$-quasi-Lyapunov, $\alpha$-equivariant and continuous map.
	Take two points $x,y$ in $\wh M_\alpha$. Let $\epsilon$ be small enough, and $\gamma$ be an $\epsilon$-pseudo-orbit from $x$ to $y$. Then any point $z$ in $\gamma$ satisfies 
	$$\min(f(x),f(y))-C\leq f(z)\leq\max(f(x),f(y))+C.$$
	In particular, $\gamma$ remains in a compact region of $\wh M_\alpha$, where the metrics on $\wh M_\alpha$ and $\wb M_\alpha$ are comparable at small scales. As a consequence, $x$ and $y$ satisfy $x\recto y$ for the flow on $\wh M_\alpha$ if and only if they satisfy $x\recto y$ for the flow on~$\wb M_\alpha$. The conclusion follows.
\end{proof}

\section{Quasi-Lyapunov with integer coefficient}\label{sec-QL-int-coeff}

We build $\alpha$-equivariant maps that are either Lyapunov, or pre-Lyapunov with a control on the recurrent set.

\subsection{Spectral decomposition: the integer case}\label{sec-quasi-L-criterion}

Denote by $\wh\Rec_\alpha$ the recurrent set inside $\wh M_\alpha$.

\begin{theorem}\label{thm-spectral-decomp}
	Let $\alpha$ be in $H^1(M,\ZZ)\setminus\{0\}$.
	If there exists an $\alpha$-equivariant pre-Lyapunov map $f\colon\wh M_\alpha\to\RR$, then $\alpha$ is quasi-Lyapunov.

	If $\alpha$ is quasi-Lyapunov, then there exists an $\alpha$-equivariant Lyapunov map $f\colon\wh M_\alpha\to\RR$, and so that $f(\wh\Rec_\alpha)$ has Lebesgue measure zero.
	When~$M$ is smooth and $\varphi$ is generated by a uniquely integrable vector field $X$, then $f$ can be chosen smooth and satisfying $df(X)<0$ outside $\wh\Rec_\alpha$.
\end{theorem}

We prove the first claim here, which is easier.

\begin{proof}[Proof of the first claim]
	Assume that for all $\epsilon>0$, there exists an $\epsilon$-pseudo-orbit $\gamma_\epsilon\colon[0,l]\to\wh M_\alpha$ with $\gamma(l)=\gamma(0)+1$. Up to replacing $\gamma_\epsilon$ by $n\cdot\gamma_\epsilon$ for some $n$ in $\ZZ$, we may assume that $\gamma_\epsilon(0)$ remains in a bounded region of $\wh M_\alpha$.
	
	When $\epsilon$ goes to zero, $\gamma_\epsilon(0)$ and $\gamma_\epsilon(1)$ convergence toward two points $x$ and $y$, which satisfies $x\recto y$ and $f(y)=f(x)+1$. It contradicts that $f$ is pre-Lyapunov. Thus, $f$ is $1$-quasi-Lyapunov, and $\alpha$ is quasi-Lyapunov. 
\end{proof}


We dedicate the rest of the section to the second and third claims in the theorem. 
Fix a class $\alpha$ in $H^1(M,\ZZ)$, assumed quasi-Lyapunov. Recall that the recurrence chains in $\wh M_\alpha$ and the one in $\wb M_\alpha$ can be identified, apart from~$\{\pm\infty\}$ (see Lemma~\ref{lem-rec-identify}).

Let us sketch the proof. Take two recurrence chains $R_1,R_2$ in $\wh M_\alpha$ so that $R_1$ is far closer to $-\infty$ than $R_2$ is. Since $\alpha$ is quasi-Lyapunov, it implies $R_1\not\recto R_2$. We build an attractor $A$ in $\wb M_\alpha$, that contains $R_1$ and $-\infty$, but not $R_2$ nor $+\infty$. Lemma~\ref{lem-preL-atr-rep} yields a pre-Lyapunov map $f\colon\wb M_\alpha\to[0,1]$ which satisfies $f\equiv 1$ on $R_2\cup\{+\infty\}$ and $f\equiv 0$ on $R_1\cup\{-\infty\}$. Then, the sum of the functions $x\mapsto f(n\cdot x)$, for~$n$ in $\ZZ$, provides an $\alpha$-equivariant pre-Lyapunov function $H\colon\wh M_\alpha\to\RR$ with $H(R_1)<H(R_2)$. By summing carefully countably many functions like $H$, for different values of $R_1,R_2$, we obtain an $\alpha$-equivariant Lyapunov function.

We build an $\alpha$-equivariant and pre-Lyapunov map. 
Building a Lyapunov map is quite more technical. So we do it in a second time. 

\begin{lemma}\label{lem-attractor-existance}
	Let $R_1,R_2$ be two disjoint compact subsets of $\wb M_\alpha$. We assume that there exists no pair $(x_1,x_2)$ in $R_1\times R_2$ that satisfies $x_1\recto x_2$. Then there exists a pair of attractor-repeller $A,A^*\subset N$ that satisfies $R_1\subset A$ and $R_2\subset A^*$. 
\end{lemma}

\begin{proof}
	By assumption, there exists $\epsilon>0$ for which there is no~$\epsilon$-pseudo-orbit from $R_1$ to $R_2$. 
	We denote by $U_\epsilon$ the open subset of points $y$ for which there is an~$\epsilon$-pseudo-orbit that starts in $R_1$ and ends at $y$. Then define $A$ to be the $\omega$-limit of $U_\epsilon$. Since for all $t\geq T$, the closure of $\psi^\alpha([t,+\infty[)$ is included in $\psi^\alpha([t-T,+\infty[)$, the set $A$ is compact. It is invariant by the flow and is the $\omega$-limit of its neighborhood $U$, so it is an attractor. It is clear that $A$ contains $R_1$ and is disjoint from $R_2$. 
\end{proof}

\begin{lemma}\label{lem-Z-equiv}
	Assume $\alpha$ quasi-Lyapunov and primitive. Let $h\colon\wb M_\alpha\to[0,1]$ be a pre-Lyapunov map which satisfies $h\equiv 0$ on a neighborhood of $-\infty$ and $h\equiv 1$ on a neighborhood of~$+\infty$. Denote by $H\colon\wh M_\alpha\to\RR$ the map defined by 
	$$H(x)=\sum_{k>0}(h(k\cdot x)-1)+\sum_{k\leq 0}h(k\cdot x).$$
	Then $H$ is well-defined, $\alpha$-equivariant and pre-Lyapunov. If $h$ is additionally smooth inside $\wh M_\alpha$, then $H$ is smooth.
\end{lemma}

\begin{proof}
	For all~$x$ within a compact region $K$ of $\wh M_\alpha$, we have $h(k\cdot x)=0$ for any small enough $k$ and $h(k\cdot x)=1$ for any large enough $k$, so the sum has finitely many non-zero terms on a neighborhood of~$x$. Thus, $H$ is well-defined and continuous. A direct computation yields 
	$$H(1\cdot x)=H(x)+1=H(x)+\alpha(1)$$ 
	for any~$x$ in $\wh M_\alpha$ and seeing $1$ as an element of $\bfrac{H_1(M,\ZZ)}{\ker_\ZZ(\alpha)}$. So $H$ is $\alpha$-equivariant. The pre-Lyapunov and smooth assertion are clear.
\end{proof}

We now start building pre-Lyapunov maps.

\begin{lemma}
	Let $\alpha$ in $H^1(M,\ZZ)\setminus\{0\}$ be quasi-Lyapunov and primitive. Then there exists an $\alpha$-equivariant pre-Lyapunov map $g\colon\wh M_\alpha\to\RR$ that sends the recurrent set in $\wh M_\alpha$ onto $\ZZ$.
\end{lemma}

\begin{proof}
	Since $\alpha$ is quasi-Lyapunov, there exist a constant $C\geq0$ and an $\alpha$-equivariant and $C$-pre-Lyapunov map $u\colon\wh M_\alpha\to\RR$. 
	Denote by $R_2\subset\wb M_\alpha$ the union of $\{+\infty\}$ and of the recurrence chains that remain in $u^{-1}([C,+\infty[)$. Also denote by $R_1$ the union of $\{-\infty\}$ and of the recurrence chains that remain in $u^{-1}(]-\infty,0])$. Since $g$ is $C$-quasi-Lyapunov, there is no pair $(x_1,x_2)$ in $R_1\times R_2$ that satisfies $x_1\recto x_2$. 
	
	We apply Lemma~\ref{lem-attractor-existance} to $N=\wb M_\alpha$. It yields a pair $A,A^*$ of attractor-repeller that satisfies $R_1\subset A$ and $R_2\subset A^*$. Then Lemma~\ref{lem-preL-atr-rep} provides us with a pre-Lyapunov map $h\colon\wb M_\alpha\to[0,1]$ which satisfies the condition of Lemma~\ref{lem-Z-equiv}. So the map $H$ defined from $h$ in the later lemma is $\alpha$-equivariant and pre-Lyapunov.
\end{proof}

From now on, we assume that $\alpha$ is primitive. The non-primitive case follows from the primitive one. 
To obtain the strongly-Lyapunov property, we sum countably many functions given above. 

Denote by $g\colon\wh M_\alpha\to\RR$ an $\alpha$-equivariant and pre-Lyapunov map constructed above, which sends the recurrent set of $\wh M_\alpha$ onto $\ZZ$.
Denote by $\Rec^+$ the union of the recurrence chains~$R$ in $\wh M_\alpha$ that satisfy $g(R)\geq 1$. Similarly denote by $\Rec^-$ the union of the recurrence chains~$R$ that satisfy $g(R)\leq -1$. Observe that for any $\alpha$-equivariant and pre-Lyapunov maps $h\colon\wh M_\alpha\to\RR$, its values on the recurrent set is determined by its values on the recurrence chains in $g^{-1}(\{0\})$. We use the sets $\Rec^\pm$ to build maps $h$ as above, so that the recurrence chains in $g^{-1}(\{0\})$ are sent onto $\{0,1\}$, with some level of control. 

Let us denote by $\mathbb{A}$ the set of attractors that contain $\Rec^-\cup\{-\infty\}$ and that are disjoint from~$\Rec^+\cup\{+\infty\}$. Lemma~\ref{lem-attractor-existance} implies that $\mathbb{A}$ is not empty.
There exists at most countably many attractors in $\wb M_\alpha$, so $\mathbb{A}$ is at most countable. 

Fix some $m\geq 1$, an attractor $A$ in $\mathbb{A}$ and $A^*$ its associated repeller. Denote by $V_m(A)$ the $\tfrac{1}{m}$-neighborhoods $A\cup A^*$. Using Lemma~\ref{lem-preL-atr-rep}, there exists a pre-Lyapunov map $h_{A,m}(A)\colon\wh M_\alpha\to[0,1]$ which satisfies $h_{A,m}\equiv0$ close to $A$, $h_{A,m}\equiv1$ close $A^*$, and $h_{A,m}$ is decreasing along the flow outside $V_m(A)$. When~$M$ is smooth $\varphi$ is generated by a continuous and uniquely integrable vector field $X$, we also assume $h_{A,m}$ smooth and that satisfies $dh_{A,m}(X)<0$ outside $V_m(A)$. 

Note that there are countably many possible choices of $A$ in $\mathbb{A}$ and $m$ in $\NN_{\geq 1}$. We enumerate them
using a sequence $A_n$ with values in $\mathbb{A}$, and a sequence of maps $h_n$ as above, so that each maps $h_{A,m}$ appears once as a $h_m$.

Let $H_n\colon M_\alpha\to\RR$ be $\alpha$-equivariant and pre-Lyapunov function given by Lemma~\ref{lem-Z-equiv} applied to $h_n$. We build a strongly-Lyapunov function as the sum $\sum_n a_nH_n$ for some scalars $a_n$ in $[0,1]$. We need technical estimations to prove that the sum has different values on different recurrence chains. 

\begin{lemma}\label{lem-L-estimate-1}
	For any recurrence chain $R\subset \wh M_\alpha$ with $g(R)=0$, and any~$n$ in $\NN$, we have $H_n(R)=h_n(R)$ and this value lies in $\{0,1\}$. 
\end{lemma}

\begin{proof}
	For any $k\geq 1$, $k\cdot R$ lies in $\Rec^+$ and $(-k)\cdot R$ lies in $\Rec^-$. Thus, we have $h(k\cdot R)=1$ and $h((-k)\cdot R)=0$. Hence, by construction of $H_n$, we have $H_n(R)=h_n(R)$ which lies in $\{0,1\}$.
\end{proof}

Recall that $\alpha$ is taken primitive, so for any recurrence chain~$R$ and $k=-g(R)$, we have $g(k\cdot R)=0$. The previous lemma has the immediate consequence.

\begin{lemma}\label{lem-L-estimate}
	For any recurrence chain $R\subset \wh M_\alpha$ and any $i,j$, the inequality $|H_i(R)-H_j(R)|\leq 1$ holds true. 
\end{lemma}

We fix a scalar $\theta$ in $]0,\tfrac{1}{2}[$, which thus satisfies $\tfrac{\theta}{1-\theta}<1$. Let $a_n^*>0$, for~$n$ in $\NN$, be any sequence that satisfies $a_{n+1}^*\leq \theta a_n^*$. We write $a^*=\sum_na_n^*$ and $a_n=\tfrac{a_n^*}{a^*}$. Note that $\sum_na_n=1$ and $a_{n+1}\leq\theta a_n$ hold for all~$n$. 
Then we define the function 
$$f=\sum_{n}a_nH_n$$
which is $\alpha$-equivariant and pre-Lyapunov. If the maps $h_n$ are smooth, and the sequence $a_k^*$ additionally goes to zero fast enough, then $f$ is smooth. 

\begin{lemma}\label{lem-spec-tech-1}
	Let $R_1,R_2\subset\wh M_\alpha$ be two recurrence chains which satisfy $f(R_1)=f(R_2)$. Then $H_n(R_1)=H_n(R_2)$ holds true for all~$n$ in $\NN$.
\end{lemma}

The lemma is an immediate consequence of the following result, applied to $n=1$.

\begin{lemma}\label{lem-injective-cantor}
	Take~$n$ in $\NN_{\geq 1}$ and $\theta<\tfrac{1}{2n+1}$. Let $(a_k)_{k\in\NN}$ be a sequence which satisfies $a_{k+1}\leq\theta a_k$ for all $k$. Then the map $h\colon\intint{-n,n}^\NN\to\RR$, defined by $h(u)=\sum_k a_ku_k$, is injective and continuous for the cylindric topology on $\intint{-n,n}^\NN$. Additionally, the image of $h$ is a Cantor set of Lebesgue measure zero.
\end{lemma}

\begin{proof}
	It is clear that $h$ is continuous. Assume that there exists two distinct sequences $u,v$ in $\intint{-n,n}^\NN$ with $h(u)=h(v)$. Let $k_0$ be the smallest integer which satisfies $u_{k_0}\neq v_{k_0}$. Then we have:
	$$a_{k_0}|u_{k_0}-v_{k_0}|
			= \left|\sum_{k>k_0} a_k(u_k-v_k)\right|
			\leq 2n\sum_{k>k_0}\theta^{k-k_0}a_{k_0}
			=\tfrac{2n\theta}{1-\theta}a_{k_0}
			<a_{k_0}$$
	which contradicts $|u_{k_0}-v_{k_0}|\geq 1$. The injectivity follows. The fact that $\im(h)$ is a Cantor set of zero measure follows from standard arguments.
\end{proof}

\begin{lemma}
	For any recurrence chain~$R$, we have $g(R)<f(R)<g(R)+1$.
\end{lemma}

\begin{proof}
	Up to replacing~$R$ by $k\cdot R$ for $k=-g(R)$, we may assume that $g(R)=0$ holds true. It follows from Lemma~\ref{lem-L-estimate-1} that $H_n(R)$ belongs to $\{0,1\}$ for all~$n$, so $f(R)=\sum_na_nH_n(R)$ belongs to $[0,1]$. 
	
	We prove now that $f(R)$ is not equal to 0. According to Lemma~\ref{lem-attractor-existance}, there exists an attractor $A_m$ in $\mathbb{A}$, whose associated repeller $A_m^*$ contains~$R$. Then the associated maps $H_m$ and $h_m$ satisfy $H_m(R)=h_m(R)=1$. It follows $\sum_na_nH_n(R)\geq a_m>0$ from Lemma~\ref{lem-L-estimate-1}. Thus, $f(R)>0$ holds. Similarly, there exists an attractor $A_p$ that contains~$R$. It follows $H_p(R)=0$. And since $\sum_na_n=1$ and $H_n(R)\leq 1$ hold for all~$n$, we have $f(R)<1$.
\end{proof}

\begin{lemma}\label{lem-strongL-eq}
	Let $R_1,R_2$ be two recurrence chains. If $f(R_1)=f(R_2)$ holds, then $R_1$ and $R_2$ are equal.
\end{lemma}

\begin{proof}
	Up to replacing $R_1$ and $R_2$ by $n\cdot R_1$ and $n\cdot R_2$ for $n=-g(R_1)$, we may assume that $g(R_1)=0$ holds true. It follows from above that $g(R_2)$ is an integer that lie in the open interval in $]f(R_1), f(R_1)+1[$. The length of this interval is one, so it contains at most one integer. It follows that we have $g(R_2)=g(R_1)=0$.

	We now reason by contradiction. Assume that $R_1$ and $R_2$ are not equal. Then up to swapping them, we may assume that $R_1\not\recto R_2$ holds. According to Lemma~\ref{lem-attractor-existance}, there exists a pair of attractor-repeller $A_n,A_n^*$ so that $A_n$ contains $\Rec^-\cup R_1$, and $A_n^*$ contains $\Rec^+\cup R_2$. Then the associated functions $H_n$ and $h_n$ satisfy $H_n(R_1)=h_n(R_1)=0$ and $H_n(R_2)=h_n(R_2)=1$. So $H_n(R_1)\neq H_n(R_2)$ holds, which contradicts Lemma~\ref{lem-spec-tech-1}.
\end{proof}

\begin{proof}[Proof of the second claim in Theorem~\ref{thm-spectral-decomp}]
	Let $\alpha$ be any quasi-Lyapunov classes in $H^1(M,\ZZ)$.
	When $\alpha$ is not primitive, and is equal to $n\beta$ for some $n\geq 1$, we have $\wh M_\alpha=\wh M_\beta$ and $f\colon\wh M_\alpha\to\RR$ is $\beta$-equivariant if and only if $nf$ is $\alpha$-equivariant. Thus, it is enough to prove the conclusion of the theorem when $\alpha$ is primitive.

	Let us now assume $\alpha$ primitive.
	We use the notations defined above in this section. The function $f\colon\wh M_\alpha\to\RR$ constructed above is $\alpha$-equivariant and constant on recurrence chains. It follows from Lemma~\ref{lem-strongL-eq} that $f$ sends two distinct recurrences chains on two distinct values. 
	
	Take a point~$x$ in $\wh M_\alpha$ outside the recurrent set. The orbit of~$x$, by the flow, converges in the future and in the past toward two recurrence chains $R_1,R_2$ in $\wb M_\alpha$. These recurrence chains are distinct since~$x$ is not recurrent. It follows that $f$ is not constant on that orbit, otherwise $f(R_1)$ and $f(R_2)$ would be equal. 
	
	We prove that $f$ is decreasing along the flow close to~$x$.
	We write $k=-g(R_1)$, so that $g(k\cdot R_1)=0$ and $g(k\cdot R_2)>0$. Lemma~\ref{lem-attractor-existance} implies that there exists an attractor $A$ in $\mathbb{A}$ that contains $k\cdot R_1$ but not $k\cdot R_2$. Then $k\cdot x$ is included neither in $A$ not in its associated repeller $A^*$. 
	Take $m\geq 1$ large enough so that $k\cdot x$ is at distance at least $\tfrac{1}{m}$ from $A\cup A$. Then the map $h_{A,m}$ constructed above is decreasing along the flow on a neighborhood of~$x$. Recall that all $h_n$ are non-increasing along the flow, and one of them is equal to $h_{A,m}$. It follows that their combination $f$ is decreasing along the flow close to~$x$. Thus, $f$ is Lyapunov.
	
	The fact that $f(\wh\Rec_\alpha)$ has Lebesgue measure zero follows from Lemma~\ref{lem-injective-cantor}.
	The smoothness conclusion follows from choosing the sequence $a_n^*$ to decrease fast enough. 
\end{proof}

\subsection{Pre-Lyapunov maps from Conley's order}\label{sec-pre-Conly-order}

We discuss here a similar result to the equivariant spectral decomposition, but with some level of control on the $\alpha$-recurrent chains. It is not a strong result in terms of dynamics, but it is very helpful for the classification of partial cross-sections.

Denote by $\wh\Rec_\alpha^*$ the set of recurrence chains, with the quotient topology. It comes with the partial order $\recto$, called the Conley's order. Note that when a continuous map $g\colon\wh M_\alpha\to\RR$ is $\alpha$-equivariant and pre-Lyapunov, it induces a continuous, $\alpha$-equivariant and map $\wh\Rec_\alpha^*\xrightarrow{g}\RR$, which is non-increasing for the Conley order.

\begin{theorem}
	Let $\alpha$ in $H^1(M,\ZZ)$ be quasi-Lyapunov, and $f\colon\wh\Rec_\alpha^*\to\ZZ$ be a non-increasing and continuous map. There exists a continuous, $\alpha$-equivariant and pre-Lyapunov map $g\colon\wh M_\alpha\to\RR$ whose restriction to $\wh\Rec_\alpha$ coincide with $f$, and that is decreasing outside $g^{-1}(\ZZ)$. When~$M$ is smooth and $\varphi$ is generated by a continuous and uniquely integrable vector field $X$, then $g$ can additionally be chosen smooth, and with $dg(X)<0$ outside $g^{-1}(\ZZ)$.
\end{theorem}

We split the proof in two cases, depending on if $\alpha$ is zero or not.

\begin{proof}[Proof in the case $\alpha\neq 0$.]
	Assume $\alpha$ primitive first.
	Let $R_1\subset\wb M_\alpha$ be the union of $\{-\infty\}$ and of the recurrence chains in $f^{-1}(\intintl{-\infty,0})$. Let $R_2$ be the union of $\{+\infty\}$ and of the recurrence chains in $f^{-1}(\intintr{1,+\infty})$. 
	By continuity of $f$, $R_1$ and $R_2$ are disjoint compact subsets of $\wb M_\alpha$.
	Since $\alpha$ is quasi-Lyapunov, we have $-\infty\not\recto R_2$ and $R_1\not\recto+\infty$, and so $R_1\not\recto R_2$ (see Lemma~\ref{lem-quasiL-infty-char}).

	It follows from Lemma~\ref{lem-attractor-existance} and Lemma~\ref{lem-preL-atr-rep} that there exists a pre-Lyapunov map $h\colon\wb M_\alpha\to[0,1]$ that satisfies $h\equiv 0$ close to $R_1$ and $h\equiv 1$ close to $R_2$. Denote by $g\colon\wh M_\alpha\to\RR$ the continuous and $\alpha$-equivariant map given in Lemma~\ref{lem-Z-equiv} applied to the map $h$.

	Let~$R$ be a recurrence chain that satisfies $f(R)=0$. For any $k>0$, we have $f(k\cdot R)\geq 0$ so $k\cdot R$ lies in $R_2$ and $h(k\cdot R)=1$ holds. Similarly, we have $h(k\cdot R)=0$ for any $k\leq0$. It follows $g(R)=0$ from the definition of $g$. Take now any recurrence chain~$R$ and $k=-f(R)$. Since $\alpha$ is primitive, we have $f(k\cdot R)=0$, and thus 
	$g(R)=g(k\cdot R)-k=f(R)$
	holds. So $g$ and $f$ coincide on $\wh\Rec_\alpha$.

	We now consider the general case on $\alpha$. We write $\alpha=n\beta$ for some $n\geq 1$ and some primitive class $\beta$ in $H^1(M,\ZZ)$. 
	For any $i$ in $\intint{0,n-1}$, we define a function $f_i\colon\wh\Rec_\alpha^*\to\ZZ$ by
	$$f_i=\left\lfloor\tfrac{f+i}{n}\right\rfloor$$
	which is $\beta$-equivariant and continuous.
	Using the Euclidian division by~$n$, it is an exercise to see that 
	$$\sum_{i=0}^{n-1}f_i=\sum_{i=0}^{n-1}\lfloor\tfrac{f+i}{n}\rfloor=f$$
	holds.
	From the primitive case, there exists a $\beta$-equivariant and pre-Lyapunov map $g_i\colon\wh M_\alpha\to\RR$ which coincides with $f_i$ in restriction to $\wh\Rec_\alpha$. We define the map $g=\sum_ig_i$, which is $\alpha$-equivariant and pre-Lyapunov. It follows from above that $g$ coincides with $f$ on $\wh\Rec_\alpha$. 

	The smoothness conclusion follows from Lemma~\ref{lem-preL-atr-rep}.
\end{proof}

We now prove the case $\alpha=0$, which is essentially already known.

\begin{proof}[Proof in the case $\alpha=0$.]
	In the case $\alpha=0$, $\wh M_\alpha$ is equal to~$M$, and $\wh\Rec_\alpha$ is simply the recurrent set of $\varphi$. When $\wh\Rec_\alpha$ is equal to~$M$, $f$ is constant, and $g=f$ suffice. Thus, we can now assume that $\varphi$ admits at least two recurrence chains.

	Up to adding a constant to $f$, we may assume that its minimum is zero. For any~$n$ in $\intint{1,\max f}$, denote by $R_n$ the union of recurrence chains~$R$ of~$\varphi$ that satisfy $f(R)<n$, and $R_n^*=\Rec\setminus R_n$. 
	It follows from Lemma~\ref{lem-attractor-existance} and Lemma~\ref{lem-preL-atr-rep} that there exists a pre-Lyapunov map $h_n\colon M\to[0,1]$ that satisfies $h_n(R_n)=0$ and $h_n(R_n^*)=1$.
	
	We define the pre-Lyapunov map $g=\sum_n h_n$. 
	For any recurrence chain~$R$ of $\varphi$, $g(R)$ is equal to the number of $n\geq 1$ that satisfy $f(R)\geq n$, that is~$f(R)$. Therefore, $g$ and $f$ coincide on the recurrent set.
	
	The smoothness conclusion follows from Lemma~\ref{lem-preL-atr-rep}.
\end{proof}

\section{Quasi-Lyapunov with real coefficients}\label{sec-general-quasiL}

We explore several questions about quasi-Lyapunov cohomology classes with real coefficients. We start with the dependence of $\Rec_\alpha$ on $\alpha$.

\subsection{Approximation by rational quasi-Lyapunov classes}

The first difficulty to study quasi-Lyapunov classes with real coefficients is that $\ker(\alpha)$ in $H_1(M,\ZZ)$ is not necessarily a co-dimension 1 sub-module of $H_1(M,\ZZ)$. So the covering $\wh M_\alpha\to M$ does not behave as well as in the integer case. Our strategy is to first approximate any quasi-Lyapunov classes by ones with rational coefficients, on which we can work more easily.

Denote by \emph{$\ker_\QQ(\alpha)$} the kernel of $\alpha$ inside $H_1(M,\QQ)$, which we identify with a subset of $H_1(M,\RR)$. 
Recall that $D_\varphi\subset H_1(M,\RR)$ is the set of asymptotic pseudo-directions of~$\varphi$.

\begin{lemma}\label{lem-ker-approx}
	Let $\alpha$ in $H^1(M,\RR)$ be quasi-Lyapunov, and take $\epsilon>0$. For any $\delta$ in $\ker(\alpha)\cap D_\varphi$, there exist finitely many scalars $a_1\cdots a_n>0$ and~$\epsilon$-pseudo-orbits $\gamma_1\cdots\gamma_n$ which satisfy $\delta=\sum_ia_i[\gamma_i]$ and $\alpha([\gamma_i])=0$ for all $i$.
\end{lemma}

\begin{proof}
	Take $\delta$ in $\ker(\alpha)\cap D_\varphi$. According to Corollary~\ref{cor-pso-decomposition}, $\delta$ is the convex combination of some $a_i[\gamma_i]$, for some $a_i>0$ and some periodic~$\epsilon$-pseudo-orbits $\gamma_i$. Since $\alpha$ is quasi-Lyapunov, we have $\alpha([\gamma_i])\geq 0$ for all $i$. And since $\alpha(\delta)=0$ holds, it follows $\alpha([\gamma_i])=0$.
\end{proof}

Denote by $\ker_\RR(\alpha)$ kernel of $\alpha$ in $H_1(M,\RR)$. The lemma above as an immediate consequence.

\begin{corollary}\label{cor-rat-ker}
	Let $\alpha$ in $H^1(M,\RR)$ be quasi-Lyapunov. Then $\ker_\RR(\alpha)\cap D_\varphi$ lies in the $\RR$-vector subspace of $H_1(M,\RR)$ spanned by $\ker_\QQ(\alpha)$.
\end{corollary}

\begin{lemma}\label{lem-quasiL-stability}
	Let $\alpha$ in $H^1(M,\RR)$ be quasi-Lyapunov. Then for any $\beta$ in $H^1(M,\RR)$ sufficiently close to $\alpha$ and with $\ker_\QQ(\alpha)\subset\ker_\QQ(\beta)$, $\beta$ is quasi-Lyapunov.
\end{lemma}

\begin{proof}
	Take $\beta$ with $\ker_\QQ(\alpha)\subset\ker_\QQ(\beta)$. By Lemma~\ref{lem-quasi-L-char}, there exists a small $\epsilon>0$ which so that $\alpha([\gamma])\leq 0$ holds for all periodic~$\epsilon$-pseudo-orbits $\gamma$. Let $\epsilon'>0$ and $\delta_1\cdots\delta_n$ be~$\epsilon$-pseudo-orbits as in the conclusion of Lemma~\ref{lem-cut-pso-finite}, that is so that every $\epsilon'$-pseudo-orbit is homologous to a sum of $\delta_i$.

	Whenever $\alpha([\delta_i])=0$ holds, $[\delta_i]$ lies in $\ker_\QQ(\alpha)$, so $\beta([\delta_i])=0$ also holds. If $\beta$ is close enough to~$\alpha$, then $\beta([\delta_i])<0$ holds for the finitely many $\delta_i$ that satisfy $\alpha([\delta_i])<0$. Hence, we have $\beta([\delta_i])\leq 0$ for all $i$. It follows from the choice of $\delta_i$ that $\beta([\gamma])\leq0$ holds for all periodic $\epsilon'$-pseudo-orbits~$\gamma$. Then using Lemma~\ref{lem-quasi-L-char}, we deduce that $\beta$ is quasi-Lyapunov.
\end{proof}

\begin{lemma}\label{lem-general-L-tech}
	Let $E$ be a $\RR$-vector space of finite dimension and $F\subset E$ be a $\QQ$-vector subspace dense inside $E$. For any~$x$ in $E\setminus F$ and any neighborhood $U\subset E$ of~$x$, there exist finitely many vectors $\beta_1\cdots\beta_p$ in $F\cap U$ and scalars $\lambda_1\cdots \lambda_p>0$ which satisfy $x=\sum_i \lambda_i\beta_i$ and so that the two families $(\beta_1\cdots\beta_p)$ and $(\lambda_1\cdots\lambda_p)$ are free over $\QQ$.
\end{lemma}

Given a $\RR$-vector space $E$ and $X\subset E$, we denote by $\conv(X)\subset E$ its convex hull.

\begin{proof}
	Take~$n$ vectors $\beta_1\cdots\beta_p$ in $F$ so that~$x$ belongs to the $\RR$-vector space spanned by the family $(\beta_i)_i$. We additionally take~$n$ minimal with this property. Then $(\beta_1\cdots\beta_p)$ is free over $\QQ$.
	
	Take $a_1\cdots a_p\in\RR$ with $x=\sum_i a_i\beta_i$. We claim that the family $(a_1\cdots a_p)$ is free over $\QQ$. If it was not the case, there exist $1\leq j\leq n$ and rational numbers $q_i$, for $i\neq j$, that satisfy:
	$$a_j=\sum_{i\neq j} q_i a_i.$$
	Note that $\beta_i + q_i\beta_j$ belongs to $F$. So the rewriting
	$$x = \sum_{i\neq j}a_i\beta_i + \left(\sum_{i\neq j} q_i a_i\right)\beta_j 
			= \sum_{i\neq j}a_i(\beta_i + q_i\beta_j)$$
	contradicts the minimality of~$n$. The claim follows.

	We may now choose $a_1\cdots a_p$ and $\beta_1\cdots\beta_p$ in $F\cap U$ as above, so that additionally~$x$ belongs to $\conv(\{0,\beta_1\cdots\beta_p\})$. Then $a_i$ is non-negative, and even positive by minimality of~$n$. 
\end{proof}

\begin{corollary}\label{cor-rat-L-span}
	Let $\alpha$ in $H^1(M,\RR)$ be quasi-Lyapunov and $U$ be a neighborhood of~$\alpha$. Then there exists some quasi-Lyapunov classes $\beta_1\cdots\beta_p$ in $H^1(M,\QQ)\cap U$ and some scalars $\lambda_1\cdots\lambda_p>0$ that satisfies $\sum_i\lambda_i\beta_i=\alpha$ and so that the two families $(\beta_1\cdots\beta_p)$ and $(\lambda_1\cdots\lambda_p)$ are free over $\QQ$.
\end{corollary}

\begin{proof}
	The set of classes $\beta$ in $H^1(M,\RR)$ with $\beta(\ker_\QQ(\alpha))=0$ is a $\RR$-vector subspaces described by rational equations, and so it admits a basis formed by elements with rational coefficients, that is in $H^1(M,\QQ)$. We apply Lemma~\ref{lem-general-L-tech} to $x=\alpha$, to a smaller neighborhood $U'\subset U$ of $\alpha$, to
	\begin{align*}
	E&=\{\beta\in H_1(M,\RR),\beta(\ker_\QQ(\alpha))=0\}, \text{ and }\\
	F&=\{\beta\in H_1(M,\QQ),\beta(\ker_\QQ(\alpha))=0\}.
	\end{align*}
	It yields some $\beta_i$ in $H^1(M,\QQ)\cap U'$ and $\lambda_i>0$ which satisfy $\alpha=\sum_i\lambda_i\beta_i$, and so that the family of $\lambda_i$ and $\beta_i$ are free over $\QQ$. If $U'$ is taken small enough, then Lemma~\ref{lem-quasiL-stability} implies that the $\beta_i$ are quasi-Lyapunov.
\end{proof}

\subsection{Cone of quasi-Lyapunov classes}

Denote by $D_\varphi^*\subset H^1(M,\RR)$ the convex cone dual to $D_\varphi$, that is the set of $\alpha$ that satisfy $\alpha(D_\varphi)\geq0$. Also denote by $\QL_\varphi$ the set of quasi-Lyapunov classes for~$\varphi$. It is a convex cone, it contains zero, and included in~$-D_\varphi^*$ (see Lemma~\ref{lem-quasiL-hom-inequality}). In general, $\QL_\varphi$ is neither open nor closed inside $-D_\varphi^*$.

We discuss a combinatorial decomposition of an arbitrary convex cone, not necessarily closed, and apply it to $\QL_\varphi$ and $-D_\varphi^*$. Denote by $\cC\subset X$ a non-empty convex cone in a finite dimensionnal $\RR$-vector space $X$.
Given a subset $E\subset X$, denote by $\langle E\rangle$ the vector subspace of $X$ spanned by~$E$. We call a \emph{open face} of $\cC$, any non-empty subset $F$ of~$\cC$ that satisfies:
\begin{enumerate}
	\item $F$ is a convex cone,
	\item $F$ is open inside $\langle F\rangle$,
	\item $F$ is maximal among the sets that satisfy the first two items. 
\end{enumerate}

Let us insist on the fact that `open' in `open face' means that it does not contains its boundary, not that it is open in $H^1(M,\RR)$.

\begin{example}
	The convex cone $cC$ contains a maximum vector subspace~$F$, which may be reduced to $\{0\}$. The set $F$ is an open face of $cC$.
\end{example}

\begin{example}
	When $-D_\varphi^*$ has a non-empty interior inside $H^1(M,\RR)$, its interior is made of classes $\alpha$ with $\alpha(D_\varphi)<0$, which are thus quasi-Lyapunov. Hence, the interior of $-D_\varphi^*$ is an open face of $\QL_\varphi$. This face, when non-empty, has the following geometric interpretation: $\alpha(D_\varphi)<0$ holds if and only if $-\alpha$ is cohomologous to a global section \cite{Fried82}.
\end{example}

We continue the section with some elementary properties satisfied by open faces, that one would expect from the choice of terminology.

\begin{lemma}\label{lem-unique-quasiL-face}
	For every~$x$ in $\cC$, there exists a unique open face of $\cC$ that contains~$x$.
\end{lemma}

\begin{corollary}
	The open faces of $\cC$ are pairwise disjoint.
\end{corollary}

\begin{proof}
	Let~$x$ be in $\cC$. Denote by $E_x\subset X$ the vector subspace generated by the set of $y$ that satisfies that the two elements $x+y$ and $x-y$ lie in $\cC$. Note that $E_x$ contains $\langle x\rangle$, for $y=\tfrac{1}{2}x$, so $E_x$ it is not empty. Denote by $F_x\subset E_x$ the interior (in $E_x$) of $\cC\cap E_x$. 

	By construction, there exists a basis $y_1\cdots y_p$ of $E_x$, so that all $x\pm y_i$ are all in $\cC$. The interior of 
	$$\conv(x+y_1,x-y_1,\cdots,x+y_p,x-y_p)$$
	lies in $\cC$, so~$x$ itself lies in $F_x$. It is clear that $F_x$ satisfies the first two items in the definition of open faces of $\cC$. 
	
	We now prove the last item. Let $F\subset\cC$ be a subset that contains~$x$ and that satisfies the first two items. Then for all $y$ in $F$ and close enough to~$x$, $x+(y-x)$ and $x-(y-x)$ are both very close to~$x$ inside $\langle F\rangle$, so they lie in~$\cC$. It follows that $y-x$ belongs to $E_x$, and so $\langle F\rangle$ is a subset of $E_x$. Therefore, $F$ is a subset of $F_x$. It follows that $F_x$ is maximal and that it is an open face of~$\cC$.
	The uniqueness follows from the same argument.
\end{proof}

From now on, we denote by $F_x$ the \emph{unique open face of $\cC$} that contains a given~$x$ in $\cC$. It follows from definition and from the same arguments as above.

\begin{corollary}\label{cor-quasiL-conv-comb}
	Let $y_1\cdots y_p$ be in $\cC$ and $E=\conv(\{y_1\cdots y_p\})$. There exists a unique open face $F$ of $\cC$ which contains the interior of $E$ (as a subset of~$\langle E\rangle$). Additionally, the $y_i$ lie in the closure of $F$. 
\end{corollary}

\begin{lemma}\label{lem-quasiL-face-inclusion}
	For any open face $F$ of $\cC$ and any $y$ in $\partial F\cap\cC$, the open face $F_y$ is included in $\partial F$.
\end{lemma}

\begin{proof}
	Assume that $F_y$ is not included inside $\partial F$ and take some $z$ in $F_y$ that is not inside $\partial F$. Then $z$ does not lie in $F$ since open faces of $\cC$ are pairwise disjoint. For some $\epsilon>0$, the elements $z_\pm=y\pm\epsilon(z-y)$ lie in $\cC$ by openness of $F_y$. So the interior of the convex cone spanned by $F_y\cup\{z_-,z_+\}$ is larger than $F$, and it satisfies the first two items in the definition of open faces of~$\cC$. It contradicts the maximality of $F$.
\end{proof}

We now go back to $\QL_\varphi$ and $-D_\varphi^*$.

\begin{lemma}\label{lem-rat-L-dense}
	For any $\alpha$ in $\QL_\varphi$ and any neighborhood $U$ of $\alpha$, there exists an element $\beta$ in $H^1(M,\QQ)\cap U$ that satisfies $F_\beta=F_\alpha$.
\end{lemma}

\begin{proof}
	It follows from Corollary~\ref{cor-rat-L-span} that there exist finitely many quasi-Lyapunov $\beta_i$ in $H^1(M,\QQ)$ and scalars $\lambda_i>0$ that satisfy $\sum_i\lambda_i\beta_i=\alpha$. 
	Take some rational $\eta_i>0$ very close to~$\lambda_i$, so that $\beta=\sum_i\eta_i\beta_i$ belongs to $U$. It follows from Corollary~\ref{cor-quasiL-conv-comb} that $F_\beta=F_\alpha$ holds true. 
\end{proof}

It has two immediate consequences:

\begin{corollary}
	For any open face $F$ of $\QL_\varphi$, $\langle F\rangle$ is given by rational equations.
\end{corollary}

\begin{corollary}
	$\QL_\varphi$ has at most countably many open faces.
\end{corollary}

Note that $-D_\varphi^*$ may itself have uncountably many open faces as illustrated in the following example. When it happens, at most countably many open faces of $-D_\varphi^*$ contain quasi-Lyapunov classes.

\begin{example}
	We take $M=\bfrac{\RR^4}{\ZZ^4}$ with the coordinate system $(x,y,z,w)$. Let $\gamma\colon\bfrac{\RR}{\ZZ}\to\bfrac{\RR^3}{\ZZ^3}$ be a parametrization of a circle $S\subset\RR^3$. We assume that $S$ is not centered at zero. Denote by $\varphi$ the following flow on~$M$. For any $w$ in $\bfrac{\RR}{\ZZ}$, in restriction to the torus $\bfrac{\RR^3}{\ZZ^3}\times\{w\}$, $\varphi$ is a linear flow with slope $\gamma(w)$. 

	Identify $H_1(M,\RR)$ with $\RR^4$. Then the set of asymptotic pseudo-directions contains $S\times\{0\}$, and it satisfies $\conv(D_\varphi)=\conv(S)\times\{0\}$. 

	Take $t$ in $\bfrac{\RR}{\ZZ}$, $\lambda$ in $\RR$, and $\alpha\colon\RR^3\to\RR$ linear that satisfies $\alpha\circ\gamma(s)\leq 0$ with equality exactly in $s=t$. Note that $\alpha+\lambda dw$ lies in $-D_\varphi^*$. We denote by $F_t$ the set of classes of the form $\alpha+\lambda dw$ as above.
	Then the uncountably many sets $F_t$ are all open faces of $D_\varphi^*$, but no element in $F_t$ is quasi-Lyapunov, as of Example~\ref{ex-linear-flow}.
\end{example}

\subsection{Dependence of $\Rec_\alpha$ on $\alpha$}\label{sec-dependence-reca}

Recall that $\Rec_\alpha$ is the $\alpha$-recurrent set of $\alpha$.
We study in this section the dependence of $\Rec_\alpha$ on $\alpha$.
Denote by $\wb F$ the closure of a set $F\subset H^1(M,\RR)$.
Recall from Lemma~\ref{lem-quasiL-face-inclusion} that if $F_1,F_2$ are two open faces of $\QL_\varphi$ so that~$F_1$ intersects $\partial F_2$, then $F_1$ is contained in $\partial F_2$.

\begin{theorem}\label{thm-dep-face}
	Assume that~$M$ is compact and $\varphi$ is a continuous flow on~$M$. 
	For any open face $F$ of $\QL_\varphi$ of $D_\varphi^*$, there exists a $\varphi$-invariant compact subset $\Rec_F\subset M$, possibly empty, which satisfy the following properties:
	\begin{enumerate}
		\item for any open face $F$ of $\QL_\varphi$ and any $\alpha$ in $F$, we have $\Rec_\alpha=\Rec_F$,
		\item for any two open faces $F_1,F_2$ of $\QL_\varphi$, we have $F_1\subset\wb F_2$ if and only if $\Rec_{F_2}\subset\Rec_{F_1}$ holds,
		\item as a consequence, $\Rec_{F_1}=\Rec_{F_2}$ implies $F_1=F_2$, and $F_1\subset\partial F_2$ implies $\Rec_{F_2}\varsubsetneq\Rec_{F_1}$,
	\end{enumerate}
\end{theorem}

We add a special case, which corresponds to the case of global sections.

\begin{theorem}\label{thm-dec-face-empty}
	With the same hypothesis, $\Rec_F$ is empty for at most one open face $F$ of $\QL_\varphi$. When it is empty, $F$ is open inside $H^1(M,\RR)$.
\end{theorem}

When $F$ is open, we do not necessarily have $\Rec_F=\emptyset$. For instance, take a flow $\varphi$ and a face $F$ which satisfies $\Rec_F=\emptyset$. Slow $\varphi$ down to obtain a flow $\psi$ that has unique fixed point $p$. Then $\QL_\psi$ is equal to $\QL_\varphi$, and the open face $F$ satisfies $\Rec_F=\{p\}$. 

Let us prove the second theorem, assuming the first one.

\begin{proof}
	The third item of Theorem~\ref{thm-dep-face} implies that $\Rec_F=\emptyset$ may hold for at most one open face.
	Assume that $\Rec_F=\emptyset$ holds. It follows from Proposition~\ref{prop-rec-empty} that $\alpha(D_\varphi)<0$ holds for any $\alpha$ in~$F$. The condition $\alpha(D_\varphi)<0$ is open in~$\alpha$, and it implies that $\alpha$ lies in $\QL_\varphi$. So $F$ is open in $H^1(M,\RR)$.
\end{proof}

We now start the proof of Theorem~\ref{thm-dep-face}. The first step is this simple observation:

\begin{lemma}\label{lem-Rec-comp}
	For any $\alpha,\beta$ in $\QL_\varphi$, we have $\Rec_{\alpha+\beta}\subset\Rec_\alpha\cap\Rec_\beta$.
\end{lemma}

The reciprocal inclusion is not true in general. 
We describe a counterexample of the reciprocal inclusion after the proof.

\begin{proof}
	Take~$x$ in $\Rec_{\alpha+\beta}$. For any small $\epsilon>0$, there exists a periodic~$\epsilon$-pseudo-orbit $\gamma$ starting on~$x$ and with $(\alpha+\beta)([\gamma])=0$. When~$\epsilon$ is small enough, we have $\alpha([\gamma])\leq0$ and $\beta([\gamma])\leq 0$ since $\alpha$ and $\beta$ are quasi-Lyapunov. It follows that $\alpha([\gamma])$ and $\beta([\gamma])$ are equal to $0$. Therefore,~$x$ lies in $\Rec_\alpha\cap\Rec_\beta$.
\end{proof}

\begin{example} 
	Take two points $p=(0,0)$ and $q=(-\tfrac{1}{2},-\tfrac{1}{2})$ in $\RR^2$ and define a $\ZZ^2$-periodic smooth vector field on $\RR^2$, illustrated in Figure~\ref{fig-rec-inclusion}, that satisfies:
	\begin{itemize}
		\item $p$ and $q$ are fixed points,
		\item there exists an orbit $\wt\gamma_0$ going from $p$ to $q$,
		\item there exists an orbit $\wt\gamma_1$ going from $q$ to $p-(0,1)$ and an orbit $\wt\gamma_2$ from $q$ to $p-(1,0)$,
		\item every other orbit is either the image of one of the $\wt\gamma_i$ by the $\ZZ^2$-action, or it goes from a point $q+(n+1,m+1)$ to the point $p+(n,m)$, and it is transverse to the lines $x+y=\cste$.
	\end{itemize}
	Take $M=\bfrac{\RR^2}{\ZZ^2}$. We consider the quotient vector field on~$M$ and its corresponding flow. We denote by $\gamma_i$ the image of $\wh\gamma_i$ in~$M$. Take two cohomology classes $\alpha=dx$ and $\beta=dy$. They are both quasi-Lyapunov. 

	The union $\wt\gamma_0\cup\wt\gamma_1$ can be approximated by pseudo-orbits from $p$ to $p-(1,0)$, so $\gamma_0$ lies inside $\Rec_{dx}$. In fact, we have $\Rec_{dx}=\{p,q\}\cup\gamma_0\cup\gamma_1$.
	Similarly, we have $\Rec_{dy}=\{p,q\}\cup\gamma_0\cup\gamma_2$.
	
	The orbit $\gamma_0$ lies in $\Rec_{dx}\cap\Rec_{dy}$, but not in $\Rec_{dx+dy}$. It is a consequence of Lemma~\ref{lem-blocking-box} applied to the value $r=\tfrac{1}{2}$ and to the map $f\colon\wh M_{dx+dy}\to\RR$ defined by $f(x,y)=x+y$.
\end{example}

\begin{figure}
    \begin{center}
        \begin{picture}(60,60)(0,0)
        \put(0,0){\includegraphics[width=60mm]{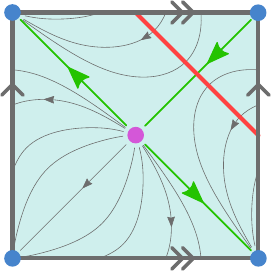}}
        \put(59.3,56){$p$}
        \put(27.8,25.5){$q$}
        \put(38,35){$\gamma_0$}
        \put(38,23){$\gamma_1$}
        \put(26,35){$\gamma_2$}
        \end{picture}
    \end{center}
    \caption{Flow on the torus which satisfies $\Rec_{dx+dy}\varsubsetneq\Rec_{dx}\cap\Rec_{dy}$.}
    \label{fig-rec-inclusion}
\end{figure}

Recall that being quasi-Lyapunov is a convex property. The main ingredient to prove the theorem is the following:

\begin{proposition}\label{prop-rec-set-inclusion}
	Let $\alpha,\beta$ be in $\QL_\varphi$. Then $\Rec_\alpha\subset\Rec_\beta$ holds if and only if for some $t>0$, the class $\alpha+t(\alpha-\beta)$ is quasi-Lyapunov.
\end{proposition}

We only prove one implication for now. The direct implication is prove further down. 

\begin{proof}[Proof of the converse implication]
	Write $\delta=\alpha+t(\alpha-\beta)$ for some $t>0$, and assume that it is quasi-Lyapunov. Then we have	$\alpha=\tfrac{t}{1+t}\beta + \tfrac{1}{1+t}\delta$. So according to Lemma~\ref{lem-Rec-comp}, $\Rec_\alpha$ is a subset of $\Rec_\beta$.
\end{proof}

To prove the direct implication, we need to characterize the $\alpha$-recurrent set when $\alpha$ has non-rational coefficient. The following lemma help up with that.

\begin{lemma}\label{lem-close-rec-is-ker}
	Let $\alpha$ be in $\QL_\varphi$. There exists a small $\epsilon>0$ so that for any~$\epsilon$-pseudo-orbit $\gamma$ in the~$\epsilon$-neighborhood of $\Rec_\alpha$, we have $\alpha([\gamma])=0$.
\end{lemma}

\begin{proof}
	We first assume that $\alpha$ has rational coefficients. Then $\wh M_\alpha$ is a $\ZZ$-covering over~$M$. Denote by $\wh\Rec_\alpha$ the recurrent set in $\wh M_\alpha$. Following Theorem~\ref{thm-spectral-decomp}, there exists an $\alpha$-equivariant Lyapunov map $f\colon \wh M_\alpha\to\RR$ so that $f(\wh\Rec_\alpha)$ has empty interior. 
	
	Let $I\subset\RR$ be a complementary component of $f(\wh\Rec_\alpha)$, which is an open interval. Choose a second open interval $J$ whose closure lies inside $I$.
	For any $\epsilon>0$, we denote by $V_\epsilon$ the~$\epsilon$-neighborhood of $\wh\Rec_\alpha$. When~$\epsilon$ is small enough, $f(V_\epsilon)$ is disjoint from $J$, and no pair of points $x,y$ in $\wh M_\alpha$ with $d(x,y)<\epsilon$ satisfies $f(x)<J$, $f(y)>J$. It follows that for any~$\epsilon$-pseudo-orbit $\gamma$ in $V_\epsilon$, $f(\gamma)$ remains in a single connected component of $\RR\setminus(J+\ZZ)$. 
	
	Let $\gamma$ be a periodic~$\epsilon$-pseudo-orbit in the~$\epsilon$-neighborhood of $\Rec_\alpha$. Its preimages in $\wh M_\alpha$ lies outside $f^{-1}(J+\ZZ)$. So $\gamma$ lifts to a periodic~$\epsilon$-pseudo-orbit that remains inside $V_\epsilon$. It follows that $\alpha([\gamma])=0$ holds.

	We now consider the general case. According to Corollary~\ref{cor-rat-L-span}, $\alpha$ can be written as a sum $\sum_i\lambda_i\beta_i$, for $1\leq i\leq n$, where the $\lambda_i$ are positive and with $\sum_i\lambda_i=1$, and the~$\beta_i$ are quasi-Lyapunov and with rational coefficients.
	When $t>0$ is small enough, $\alpha+t(\alpha-\beta_i)$ lies in $E$, so it is quasi-Lyapunov. Therefore, we have $\Rec_\alpha\subset\Rec_{\beta_i}$ for all $i$ (see the converse implication of Proposition~\ref{prop-rec-set-inclusion}). 
	
	We can apply the rational case to $\beta_i$. It yields some $\epsilon_i>0$ so that every periodic $\epsilon_i$-pseudo-orbit $\gamma$ at distance at most $\epsilon_i$ from $\Rec_\beta$ satisfies $\beta_i([\gamma])=0$. Let $\epsilon>0$ be the minimum of the $\epsilon_i$, and $\gamma$ be a periodic~$\epsilon$-pseudo-orbit in the~$\epsilon$-neighborhood of $\Rec_\alpha$. It follows from $\Rec_\alpha=\Rec_{\beta_i}$ that $\beta_i([\gamma])=0$ holds for all~$i$. Hence, $\alpha([\gamma])=0$ holds true. 
\end{proof}

\begin{proof}[Proof of the direct implication in Proposition~\ref{prop-rec-set-inclusion}]
	We reason by contraposition. Assume that for all $t>0$, $\alpha+t(\alpha-\beta)$ is not quasi-Lyapunov. 
	
	We prove by contradiction that for any $\epsilon>0$ small enough, there exists a periodic~$\epsilon$-pseudo-orbit $\gamma_\epsilon$ that satisfies $\alpha([\gamma_\epsilon])=0$ and $\beta([\gamma_\epsilon])<0$. Let us assume that it does not hold.
	According to Lemma~\ref{lem-cut-pso-finite}, there exists $\epsilon'>0$ and finitely many periodic~$\epsilon$-pseudo-orbit $\delta_i$ that satisfies that any $\epsilon'$-pseudo-orbit is homologous to a sum $\sum_ia_i\delta_i$ with $a_i>0$. 
	Since $\alpha$ and $\beta$ are quasi-Lyapunov, when~$\epsilon$ is small enough, we have $\alpha([\delta_i])\leq0$ and $\beta([\delta_i])\leq0$ for all $i$.
	From our last assumption, we have either $\alpha([\delta_i])<0$ or $\beta([\delta_i])=0$ for any $i$. It follows that
	$$(\alpha+t(\alpha-\beta))([\delta_i])\leq0$$
	holds true for all $i$ and for all $t$ small enough. So $(\alpha+t(\alpha-\beta))(D_{\varphi,\epsilon',2T})\leq 0$ holds, and $ \alpha+t(\alpha-\beta)$ is quasi-Lyapunov (see Lemma~\ref{lem-quasi-L-char}), which contradicts our first assumption. The claim follows.

	Denote by $\gamma_\epsilon$ a periodic~$\epsilon$-pseudo-orbit that satisfies $\alpha([\gamma_\epsilon])=0$ and $\beta([\gamma_\epsilon])<0$. We claim that when~$\epsilon$ is small enough, $\gamma_\epsilon$ remain arbitrarily close to $\Rec_\alpha$. To prove the claim, assume that there exist $\eta>0$ and arbitrarily small $\epsilon>0$ for which $\gamma_\epsilon$ does not remain in the $\eta$-neighborhood of $\Rec_\alpha$. When~$\epsilon$ goes to zero, $\gamma_\epsilon$ accumulates on a point $z$ in~$M$ at distance at least $\eta$ from~$\Rec_\alpha$. The point $z$ is $\alpha$-recurrent by construction (recall that $\alpha([\gamma_\epsilon])=0$). So $z$ lies in $\Rec_\alpha$, which contradict the line above. 

	Therefore, $\gamma_\epsilon$ remains arbitrarily close to $\Rec_\alpha$ when~$\epsilon$ is small enough. If we had $\Rec_\alpha\subset\Rec_\beta$, then Lemma~\ref{lem-close-rec-is-ker} would implies $\beta([\gamma_\epsilon])=0$ for all small enough~$\epsilon$. By construction, we have $\beta([\gamma_\epsilon])<0$ for all~$\epsilon$, so $\Rec_\alpha$ is not included in $\Rec_\beta$. The conclusion follows by contraposition.
\end{proof}

\begin{proof}[Proof of Theorem~\ref{thm-dep-face}]
	Let $\alpha$ and $\beta$ be in $\QL_\varphi$.
	We first assume that $F_\beta\subset\wb F_\alpha$ holds. By openness of $F_\alpha$, there exists $t>0$ for which $\alpha+t(\alpha-\beta)$ belongs to $F_\alpha$. Then Proposition~\ref{prop-rec-set-inclusion} implies $\Rec_\alpha\subset\Rec_\beta$. 

	It follows that when $F_\alpha=F_\beta$ holds, we have $\Rec_\alpha=\Rec_\beta$.
	Therefore, for any open face $F$ of $\QL_\varphi$, we can set $\Rec_F=\Rec_\alpha$ for any $\alpha$ in $F$. The first items of the theorem follows.
	
	Assume now that we have $\Rec_\alpha\subset\Rec_\beta$. By Proposition~\ref{prop-rec-set-inclusion}, there exists $t>0$ for which $\alpha+t(\alpha-\beta)$ is quasi-Lyapunov. Then $\alpha$ lies in the interior of the set $\conv(\{\beta,\alpha+t(\alpha-\beta)\})$. It follows from Corollary~\ref{cor-quasiL-conv-comb} that $\beta$ lies in $\wb F_\alpha$, or equivalently $F_\beta\subset\wb F_\alpha$. The second and third items of the theorem follow.
\end{proof}

\subsection{Spectral decomposition: the real case}\label{sec-spect-decomp-real}

We prove a similar result to Theorem~\ref{thm-spectral-decomp} for quasi-Lyapunov classes with real coefficients. Recall that $\wh\Rec_\alpha$ is the recurrent set inside $\wh M_\alpha$.

\begin{theorem}\label{thm-quasiL-map-non-rat}
	For any quasi-Lyapunov class $\alpha$ in $H^1(M,\RR)$, there exists an $\alpha$-equivariant and strongly-Lyapunov map $f\colon\wh M_\alpha\to\RR$, for which the Lebesgue measure of $f(\wh\Rec_\alpha)$ is zero. 
	When~$M$ is smooth and $\varphi$ is generated by a uniquely integrable vector field $X$, then $f$ can be chosen smooth and satisfying $df(X)<0$ outside $\wh\Rec_\alpha$.
\end{theorem}

We put the emphasis on the `strongly'-Lyapunov property, since in the real coefficients case, it is a priori a strictly stronger than the Lyapunov property.


Contrary to the integer coefficients case, we do not have a converse. 
The lack of compactness of the pre-images $f^{-1}(t)$ prevent us from applying the same arguments as in the integer coefficients case. 
In fact, there exists a pre-Lyapunov map for some non-quasi-Lyapunov classes. We do not know any counterexample with a strongly-Lyapunov map.

\begin{example}
	Let $M=\bfrac{\RR^2}{\ZZ^2}$ be the standard torus, $\varphi$ be a linear flow with irrational slope $(u,v)$, and $\alpha=vdx-udy$. Then we have $\wh M_\alpha=\RR^2$. The lifted flow on $\wh M_\alpha$ is linear, so its recurrent set is empty.
	Thus, the map $f\colon\wh M_\alpha\to\RR$ defined by $f(x,y)=vx-udy$ is pre-Lyapunov and $\alpha$-equivariant, but $\alpha$ is not quasi-Lyapunov.
\end{example}

Before proving the theorem, let us explain why trying to build $f$ by hand, as in the integer coefficient case, does not work. When $\alpha$ has integer coefficients, $\wh M_\alpha$ admits a co-compact action by $\bfrac{H^1(M,\ZZ)}{\ker_\ZZ(\alpha)}\simeq\ZZ$. So it can be compactified by adding two points $\pm\infty$ at infinity. The flow extends on the compactification. A key point is that the boundary at infinity is disconnected, and that $-\infty\not\recto+\infty$ holds, so we can build attractors that contains $-\infty$ but not $+\infty$. When $\alpha$ does not have rational coefficients,~$\bfrac{H^1(M,\ZZ)}{\ker_\ZZ(\alpha)}$ has a higher rank than $\ZZ$, so no good compactification allows a generalization of the arguments. A natural compactification would be to add the visual boundary of $\bfrac{H^1(M,\ZZ)}{\ker_\ZZ(\alpha)}$ to $\wh M_\alpha$. Then the flow extends as a constant flow on the visual boundary. The visual boundary is connected, so the flow is recurrent on it. Hence, we can not choose an attractor that contains a part of the visual boundary but not all of it.
To compensate this issue, we use the tools developed in the previous section.

We now start the proof. 
According to Corollary~\ref{cor-rat-L-span}, there exist finitely many quasi-Lyapunov classes $\beta_1\cdots\beta_p$ in $H^1(M,\QQ)$, and $\lambda_1\cdots\lambda_p>0$ which satisfy $\alpha=\sum_i\lambda_i\beta_i$ and so that $(\beta_1\cdots\beta_p)$ and $(\lambda_1\cdots\lambda_p)$ are free of $\QQ$. Up to choosing $\beta_i$ closer to $\alpha$, we may assume that $\beta_i$ lies in $F_\alpha$. Up to replacing $\beta_i$ and $\lambda_i$ by $r_i^{-1}\beta_i$ and $r_i\lambda_i$ for the unique positive rational $r_i$ that satisfies $\beta_i(H_1(M,\ZZ))=r_i\ZZ$, we may assume that $\beta_i$ has integer coefficients, and that it is primitive. 

In that section, we fix $p\geq 1$, $\beta_i$ in $H_1(M,\ZZ)$ and $\lambda_i>0$ for $p$ in $\intint{1,p}$ which satisfies the following:
\begin{itemize}
	\item $\alpha=\sum_{i=1}^p\lambda_i\beta_i$, 
	\item $\beta_i$ lies in $F_\alpha$ and it is primitive,
	\item the two families $(\beta_i)_i$ and $(\lambda_i)_i$ are free over $\QQ$.
\end{itemize}

We build a strongly-Lyapunov and $\alpha$-equivariant function in several steps. The first step is to build $\beta_i$-equivariant functions.

\begin{lemma}
	For all $1\leq i\leq n$, we have $\ker_\QQ(\alpha)\subset\ker_\QQ(\beta_i)$.
\end{lemma}

\begin{proof}
	If the inclusion did not hold, there would exist some~$x$ in $\ker_\QQ(\alpha)$ outside $\ker_\QQ(\beta_i)$ for some $i$. Then we would have 
	$$\sum_{i=1}^n\lambda_i\beta_i(x)=\alpha(x)=0$$
	where $(\beta_1(x)\cdots\beta_p(x))$ is not identically zero. It contradicts that $(\lambda_1\cdots\lambda_p)$ is free over $\QQ$.
\end{proof}

Recall that $\wh M_\alpha$ is the quotient of $\wh M$ by the action of $\ker_\ZZ(\alpha)$. It follows from the lemma that there is a natural covering map $\pi_i\colon\wh M_\alpha\to\wh M_{\beta_i}$. 
To prove Theorem~\ref{thm-spectral-decomp}, we build $\beta_i$-equivariant and Lyapunov map $f^i\colon\wh M_{\beta_i}\to\RR$. 
The composition $f_i\circ\pi_i\colon\wh M_\alpha\to\RR$ is clearly $\beta_i$-equivariant and pre-Lyapunov. Note that it is not necessarily strongly-Lyapunov. 
We define the map 
$$f=\sum_{i=1}^n\lambda_i f_i\circ\pi_i,$$
which is smooth, $\alpha$-equivariant and pre-Lyapunov. 
Proving that $f$ is strongly-Lyapunov require more work. The first step is the following lemma. The second step requires reusing the construction from Section~\ref{sec-quasi-L-criterion}.

By assumption, $\beta_i$ lies inside $F_\alpha$. So we have $\Rec_{\beta_i}=\Rec_\alpha$ for all~$i$ (see Theorem~\ref{thm-dep-face}). The recurrence chains are the connected components of the recurrence set, so a set in~$M$ is an $\alpha$-recurrence chain if and only it is a $\beta_i$-recurrence chains. It has the following consequence.

\begin{lemma}\label{lem-strongL-step1}
	For any points $x,y$ in $\wh M_\alpha$ with $x\recto y$ and $y\not\recto x$, we have $f(x)>f(y)$. 
\end{lemma}

\begin{proof}
	We prove the lemma by contraposition. Take $x,y$ in $\wh M_\alpha$ that satisfy $x\recto y$ and $f(x)=f(y)$. For all $i$, we have $f_i\circ\pi_i(x)\geq f_i\circ\pi_i(y)$, and there sum are equal. So $f_i\circ\pi_i(x)=f_i\circ\pi_i(y)$ holds for all $i$. 
	
	Since $f_i$ is strongly-Lyapunov, it implies that $\pi_i(x)$ and $\pi_i(y)$ are in the same recurrence chain inside $\wh M_{\beta_i}$. This recurrence chain projects onto a $\beta_i$-recurrence chain inside~$M$, which is also an $\alpha$-recurrence chain as discuss above.
	Let~$R$ be the recurrence chain in $\wh M_\alpha$ that contains~$x$. It follows from above that there exists $\delta$ in $H_1(M,\ZZ)$ which satisfies that $\delta\cdot y$ lies in~$R$. Since $f$ is $\alpha$-equivariant, we have 
	$$f(y)=f(x)=f(R)=f(\delta\cdot y)=\alpha([\delta])+f(y)$$
	which implies $\alpha([\delta])=0$. Since $\wh M_\alpha$ is obtained after a quotient by~$\ker_\ZZ(\alpha)$, we have $\delta\cdot y=y$. Thus,~$x$ and $y$ are in the same recurrence chain, and $y\recto x$ holds. The conclusion follows.
\end{proof}

It remains to prove one point: $f$ sends two distinct recurrence chains onto two distinct values. The difficulty is that this property does not behave well under sum.
To obtain this property, we need to choose the functions $f_i$ more carefully. 

Let us recall the construction of $f_i$ from the Section~\ref{sec-quasi-L-criterion}. 
We built countably many $\beta_i$-equivariant and pre-Lyapunov maps $H^i_k\colon\wh M_{\beta_i}\to\RR$, for all~$k$ in~$\NN$, which send the recurrent set of $\wh M_{\beta_i}$ inside $\ZZ$. Additionally, we have $|H^i_k-H^i_{k'}|\leq 1$ for all $i,k,k'$.
Fix some $\theta$ in $]0,\tfrac{1}{2}[$ so that $\tfrac{2\theta}{1-\theta}<1$ holds.
Then we built a sequence $a_k>0$ with $\sum_k a_k=1$, and $a_{k+1}\leq\theta a_k$. We then prove that the sum $\sum_ka_kH^i_k$ is $\wh\beta_i$-equivariant and strongly-Lyapunov. The map $f_i$ is obtained as $f_i=\sum_ka_kH^i_k$. 
The fact that $f_i$ sends two distinct recurrence chains on two distinct values is a consequence of the fast decreasing assumption on $a_k$. 

Let $a_k>0$ and $H_k^i$ be as above. Note that we choose $a_k$ independent on~$i$. Then we have by definition of $f$: 
$$f=\sum_{i=1}^p\lambda_ia_kH^i_k\circ\pi_i.$$
Let us rewrite $f$ in a form that we can control.
We write $\wt H^i_k=H^i_k\circ\pi_i$ and $\wt H^i_{k,0}=\wt H^i_k-\wt H^i_0$. By construction, we have 
\begin{align*}
	f 	&= \sum_{i=1}^p\sum_{k\geq 0}\lambda_ia_k\wt H^i_k \\
	 	&= \sum_{i=1}^p\lambda_ia_0\wt H^i_0 + \sum_{i=1}^p\sum_{k\geq 1}\lambda_ia_k(\wt H^i_{k,0}+\wt H^i_0) \\
		&= \sum_{i=1}^p\lambda_i \wt H^i_0 + \sum_{i=1}^p\sum_{k\geq 1}\lambda_ia_k\wt H^i_{k,0}.
\end{align*}
So given two recurrence chains $P,R$, we have the following equation:
\begin{equation}\label{eq-strongL-cond}
	f(P)-f(R)= \sum_{i=1}^p\lambda_i (\wt H^i_0(P)-\wt H^i_0(R)) + \sum_{i=1}^p\sum_{k\geq 1}\lambda_ia_k(\wt H^i_{k,0}(P)-\wt H^i_{k,0}(R))
\end{equation}
The left sum remains inside the countable set $\sum_i\lambda_i\ZZ$. After choosing $a_k$ well, the right sum will remain in a Cantor set of zero Lebesgue measure, which contains $0$. So it is likely that these two sets intersects only on $0$. If this holds true, we immediately deduce when $f(P)$ and $f(R)$. Then $P=R$ follows quickly. 

To control the right sum, we choose the values of $\lambda_i a_k$ to decrease fast enough. The first step is to control $\lambda_i$.

\begin{lemma}
	Given $\theta>0$, there exist finitely many primitive classes $\wh\beta_i$ in $H_1(M,\ZZ)\cap F_\alpha$ and scalars $\wh\lambda_i>0$, with $i$ in $\intint{1,n}$ which satisfy:
	\begin{itemize}
		\item $\alpha=\sum_{i=1}^n\wh\lambda_i\wh\beta_i$,
		\item the two families $(\wh\beta_i)_i$ and $(\wh\lambda_i)_i$ are free over $\QQ$,
		\item for all $i$, we have $\wh\lambda_{i+1}<\theta\wh\lambda_i$.
	\end{itemize}
\end{lemma}

\begin{proof}
	Fix some $\eta>0$ smaller than all the ratios $\tfrac{\lambda_i}{\lambda_j}$.
	Take some elements $\beta_i'$ in $H_1(M,\QQ)\cap F_\alpha$ and very close to $\beta_i$, and $\lambda_i'>0$, very close to $\lambda_i$, so that $\alpha=\sum_i\lambda_i'\beta_i'$ is satisfied. We choose them so that $\tfrac{\lambda_j'}{\lambda_i'}\geq \eta$ holds for all $i,j$.
	Write $\beta_i'=r_i'\wh\beta_i$ where $r_i'$ is a positive rational, and where $\wh\beta_i$ lies in $H_1(M,\ZZ)$ and is primitive. Also write $\wh\lambda_i=\lambda_i'r_i'$, so that $\alpha=\sum_i\wh\lambda_i\wh\beta_i$ holds true.
	
	We claim that we can choose $\beta_i'$ so that $r_i'$ is arbitrarily small. For that, take a very large ball $B$ inside $H^1(M,\ZZ)$. It is finite, so we may choose $\beta_i'$ outside the finite union of lines $\RR\cdot B$. Then $\wh\beta_i$ does not lie in $B$, so it is very large, and $r_i'$ is very small.

	Lastly, by choosing $\beta_i'$ inductively on $i$, we may assume that $r_{i+1}'\leq\eta\theta r_i'$ holds for all $i$. Then we have 
	$$\wh\lambda_{i+1}=\lambda_{i+1}'r_{i+1}'\leq\lambda_{i+1}'\eta\theta r_i'\leq\tfrac{\lambda_{i+1}'}{\lambda_i'}\eta\theta\wh\lambda_i\leq\theta\wh\lambda_i$$
	which concludes the proof.
\end{proof}

Up to replacing $\beta_i$ by $\wh\beta_i$ from the previous lemma, we choose some $\theta$ in~$]0,\tfrac{1}{3}[$ and assume that $\lambda_{i+1}<\theta_i\lambda_i$ holds for all $i$. 

Let use discuss the idea of the rest of the proof. Recall that $p$ is the number of $\beta_i$ we are given. We write $b_n=\lambda_ia_k$ where $n=pk+i$ is the Euclidian division of~$n$ by $p$.
In Equation~\ref{eq-strongL-cond} the right sum can be written as the sum $\sum_{n\geq p}b_nu_n$ where $u_n=\wt H^i_{k,0}(R_1)-\wt H^i_{k,0}(R_2)$ lies in $\intint{-2,2}$. 
If the sequence $b_n$ satisfies $b_{n+1}<\tfrac{1}{3}b_n$, then it follows from Lemma~\ref{lem-injective-cantor} that the sum $\sum_{n}b_nu_n$ lies in a Cantor set. 

We can describe that Cantor set as follows. Let $C_0$ be an interval centered at $0$. We inductively build a compact set $C_{n+1}$ by removed to $C_n$ the union of $4^n$ intervals of the same size. Note that the size of the removed interval is given by $b_n$. 

Let us enumerate the elements $\sum_i\lambda_i\ZZ\setminus\{0\}$ as a sequence $x_n$ for $n\geq 1$. Then we choose $C_n$ to have a very small Lebesgue measure, so that $x_n$ lies outside $C_n$. Doing it inductively yield a Cantor set $\cap_nC_n$ which is disjoint from $\sum_i\lambda_i\ZZ\setminus\{0\}$. At the step~$n$, we need to ensure that $\partial C_n$ lies outside $\sum_i\lambda_i\ZZ\setminus\{0\}$, since once $b_n$ fixed, we may not be able to disjoint $\partial C_m$ and~$x_m$ for $m>n$.

We now start the technical proof. 

\begin{lemma}\label{lem-precise-Cantor}
	For any $\eta>0$ and any sequence $\theta_n>0$ for $n\geq 1$, there exists a sequence $a_k>0$, for $k\geq 0$, that satisfies the following:
	\begin{enumerate}
		\item $\sum_ka_k=1$, 
		\item $a_{k}\leq\eta a_{k-1}$ and $a_k\leq\theta_k$ hold for all $k\geq 1$,
		\item for any sequence $u_{i,k}$, for $1\leq i\leq p$, $k\geq 1$, with value in $\intint{-2,2}$, the sum 
		$$\sum_{i=1}^p\sum_{k\geq 1}\lambda_ia_ku_{i,k}$$
		does not lie in $\sum_i\lambda_i\ZZ\setminus\{0\}$.
	\end{enumerate}
\end{lemma}

\begin{proof}
	We build inductively the sequence $a_n>0$, for~$n$ in $\NN$, so that the following holds. Write $\wb\lambda=\sum_i\lambda_i$. At the step~$n$, we set the value of $a_n$. Then we define 
	$$\wb a_n=1-\sum_{k=1}^na_k$$
	and the subset $K_0=[-2\wb\lambda\wb a_0,2\wb\lambda\wb a_0]$ and
	$$K_n=\left\{\sum_{i=1}^p\sum_{k=1}^n\lambda_ia_ku_{i,k} + \wb\lambda\wb a_nv \;\middle|\; \begin{array}{l}
		u_{i,k}\in\intint{-2,2},\\
		v\in[-2,2]
	\end{array}\right\}.$$
	Note that $K_n$ corresponds to the compact set $C_{pn}$ described informally above, whose intersection is the Cantor set.
	Denote by $x_n$, for~$n$ in $\NN$, the elements in $\sum_i\lambda_i\ZZ\setminus\{0\}$.
	We will build $a_n$ so that the three following conditions are satisfied:
	\begin{enumerate}
		\item the family of the $\lambda_i$ and of the $\lambda_ia_k$ is free over $\QQ$, 
		\item $x_n$ is not in $K_n$,
		\item $0<\wb a_n\leq \min(\theta_{n+1},\tfrac{a_n}{\eta})$.
	\end{enumerate}

	Let us now start to build the sequence $a_n$. 
	At the step~$n$, if $a_n$ is chosen generic, then the families of $\lambda_i$ and of the $\lambda_ia_k$, for $k\leq n$ is free over $\QQ$. So choosing $a_n$ inductively generic ensure that the first item is satisfied.

	For $n=0$, $K_0$ is an interval centered at zero, of length $4\wb\lambda(1-a_0)$, so taking $a_0$ close enough to 1 ensure that $x_0$ does not lie in $K_0$. Similarly, the third item is satisfied when $a_0$ is close enough to 1.
	

	Take $n\geq 1$ so that $a_{n-1}$ has been built, satisfying the three conditions.
	Note that the Lebesgue measure of $K_n$ is no more than 
	$$\Leb(K_n)\leq 5^{pn}\times 4\wb\lambda\wb a_n$$
	which can be chosen arbitrarily small if $a_n$ is chosen close enough to $\wb a_{n-1}$. 

	Set temporally $a_n=\wb a_{n-1}$, so that $\wb a_n=0$ holds and $K_n$ is finite. Any element $y$ in~$K_n$ can then be written as: 
	\begin{align*}
		y	&=\sum_{i=1}^p\sum_{k=1}^n\lambda_ia_ku_{i,k} \\
			&=\sum_{i=1}^p\sum_{k=1}^n\lambda_ia_k(u_{i,k}-u_{i,n})+\sum_{i=1}^p\sum_{k=1}^n\lambda_ia_ku_{i,n} \\
			&=\sum_{i=1}^p\sum_{k=1}^{n-1}\lambda_ia_k(u_{i,k}-u_{i,n})+(1-a_0)\sum_{i=1}^p\lambda_iu_{i,n} 		
	\end{align*}
	with $u_{i,k}$ in $\intint{-2,2}$. Assume that such an element $y$ is equal to $x_n$. The first induction hypothesis, that is the freedom of $(\lambda_i,\lambda_ia_k)_{i,k}$, implies that $u_{i,n}=0$ and $u_{i,k}-u_{i,n}=0$ holds for all $i$ and $1\leq k<n$. It yields $y=0$, which contradicts that fact that $x_n$ is not zero. Thus $x_n$ lies at a positive distance from $K_n$.
	
	Notice that $K_n$ is continuous in $a_n$. So if $a_n$ is smaller and very close to~$\wb a_{n-1}$, then $K_n$ is disjoint from $x_n$. Choosing $a_n$ generic and even closer to~$\wb a_{n-1}$ ensure that the first and third points are satisfied.
	
	By induction, there exists a sequence $a_n$ that satisfies the three conditions above.
	Up to making $\theta_n$ smaller, we may assume that $\wb a_n$ converges toward zero, so that $\sum_ia_i=1$ holds. Then the first and the second items in the conclusion of the lemma are satisfied.

	It remains to prove the third item: no sum of the form $\sum_i\sum_k\lambda_ia_ku_{i,k}$ lies in $\sum_i\lambda_i\ZZ\setminus\{0\}$.
	Take a sequence $u_{i,k}$ in $\intint{-2,2}$ for $k\geq 1$, and write
	$$x=\sum_{i=1}^p\sum_{k\geq1}\lambda_ia_ku_{i,k}.$$
	We claim that for all~$n$,~$x$ belongs to $K_n$. Indeed, we have:
\begin{align*}
	\left|x-\sum_{i=1}^p\sum_{k=1}^n\lambda_ia_ku_{i,k}\right|
		&\leq\sum_{i=1}^p\sum_{k>n}\left|\lambda_ia_ku_{i,k}\right| \\
		&\leq 2\wb\lambda\wb a_n
\end{align*}
so~$x$ belongs to the interval
$$\sum_{i=1}^p\sum_{k=1}^n\lambda_ia_ku_{i,k}+[-2\wb\lambda\wb a_n,2\wb\lambda\wb a_n],$$
which is a subset of $K_n$.
It follows from the second condition that~$x$ is different from $x_n$. Therefore,~$x$ does not lie in $\sum_i\lambda_i\ZZ\setminus\{0\}$.
\end{proof}

\begin{proof}[Proof of Theorem~\ref{thm-quasiL-map-non-rat}]
	We reuse the elements $\beta_i,\lambda_i,H^i_k$ and the notations $\wt H^i_k,\wt H^i_{k,0}$ defined above.
	Choose $\eta$ in $]0,\tfrac{1}{3}[$. 
	When~$M$ is smooth and $\varphi$ is generated by a continuous and uniquely integrable vector fields, the maps $H^i_k$ can be chosen smooth. In that case, we take a sequence $\theta_n>0$ which satisfies that all sequences $(a_n)_n$ with $|a_n|\leq\theta_n$, the $p$ functions $\sum_ka_kH^i_k$ are smooth. In the other case, we set $\theta_n=1$. 
	
	Let $(a_n)_n$ be the sequence given by Lemma~\ref{lem-precise-Cantor}, and $f=\sum_i\lambda_ia_iH^i_k\circ\pi_i$ be the corresponding function.
	It follows that $f$ is $\alpha$-equivariant and pre-Lyapunov. Lemma~\ref{lem-strongL-step1} ensure that for any $x,y$ in $\wh M_\alpha$ that satisfy $x\recto y$ and $y\not\recto x$, we have $f(x)>f(y)$.
	
	It remains to prove that $f$ takes distinct values on distinct recurrence chains.
	Take two recurrence chains $P,R$ in $\wh M_\alpha$ with $f(P)=f(R)$. Recall Equation~\ref{eq-strongL-cond}:
	\begin{equation}\label{eq-bis}
		f(P)-f(R)= \sum_{i=1}^p\lambda_i (\wt H^i_0(P)-\wt H^i_0(R)) + \sum_{i=1}^p\sum_{k\geq 1}\lambda_ia_k(\wt H^i_{k,0}(P)-\wt H^i_{k,0}(R)).
	\end{equation}
	The third item in Lemma~\ref{lem-precise-Cantor} and the Equation~\ref{eq-strongL-cond} imply that the two sums on the right are equal to zero. Since the family $(\lambda_i)_i$ is free over $\QQ$, we have $\wt H^i_0(P)=\wt H^i_0(R)$ for all $i$. Applying Lemma~\ref{lem-injective-cantor} to the second sum implies that $\wt H^i_{k,0}(P)=\wt H^i_{k,0}(R)$ for all $i,k$, and thus $\wt H^i_k(P)=\wt H^i_k(R)$ holds. 
	
	Recall that $\Rec_{\beta_i}=\Rec_\alpha$ holds by construction of $\beta_i$. So the image by $\wh M_\alpha\xrightarrow{\pi_i} M_{\beta_i}$ of $P$ and $Q$ are recurrence chains. It follows from above that the image by $f^i=\sum_ka_kH^i_k$ of the two recurrences chains $\pi_i(P)$ and $\pi_i(R)$ are equal.
	According to Lemma~\ref{lem-strongL-eq}, $\pi_i(P)$ and $\pi_i(R)$ are equal. 
	Hence, $P$ and~$R$ lie above the same $\alpha$-recurrence chain. So there exists $\delta$ in $H_1(M,\ZZ)$ that satisfies $\delta\cdot P=R$. The map $\pi_i$ is $H_1(M,\ZZ)$-equivariant, so $\beta_i(\delta)=0$ holds for all $i$. Hence, we have $\alpha(\delta)=\sum_i\lambda_i\beta_i(\delta)=0$. So $\delta$ lies in $\ker_\ZZ(\alpha)$. Together with $\delta\cdot P=R$, it implies $P=R$. It concludes the proof that $f$ is strongly-Lyapunov.

	The fact that the image by $f$ of the recurrent set has zero Lebesgue measure is a consequence of Equation~\ref{eq-bis} and Lemma~\ref{lem-injective-cantor}.

	The smoothness conclusion follows from the construction.
\end{proof}

\appendix

\section{Equivalence of asymptotic directions}\label{app-asymptotic-cone}

In this appendix, we prove that the set of asymptotic directions, defined by Fried \cite{Fried82}, and the set of asymptotic pseudo-directions, defined in Section~\ref{sec-ass-ps-dir}, span the same convex set. We additionnal prove that this convex set is given by the cohomology classes of invariant probability measures. Let us first introduce the sets in play.

Fried worked under some regularity assumptions, that is on smooth Riemann manifolds and with $\Class^1$ flows. In order to remains general, we will rephrase his definition with less regularity. Let $B$ be a cover of~$M$ by open, connected and contractible subsets. Recall that $\wh M\xrightarrow{\wh\pi}M$ is the universal Abelian covering of~$M$. A continuous curve $\gamma\colon[0,1]\to M$ is said \emph{$B$-bounded} if given a lift $\wt\gamma$ of $\gamma$ in $\wh M$, for any $U$ in $B$, $\wt\gamma$ intersects at most one lift of $U$ in $\wh M$. We let the reader verify the following two facts:
\begin{enumerate}
	\item Any two points in~$M$ are connected by a $B$-bounded curve.
	\item There exists a finite set $E$ in $H_1(M,\ZZ)$, so that for any two $B$-bounded curves $\gamma$ and $\delta$ that share their end points, the homology class of $\gamma\cup\delta$ lies in $E$. 
\end{enumerate}

From now on, we fix a cover $B$ as above.
For~$x$ in~$M$, denote by $\gamma_{x,t}$ a closed curve obtained as the concatenation of the orbit arc $\varphi_{[0,t]}(x)$ of length $t$, with a $B$-bounded curve from $\varphi_t(x)$ to~$x$.

\begin{definition}
	Let $D^F_\varphi\subset H_1(M,\RR)$, $F$ for Fried, be the set of accumulation points of the homology classes $\tfrac{1}{t}[\gamma_{x,t}]$ for~$x$ in~$M$ and $t$ that goes $+\infty$. Or more precisely:
	$$D^F_\varphi=\bigcap_{T>0}\clos\left(\set{\tfrac{1}{t}[\gamma_{x,t}]\in H_1(M,\RR),x\in M, t\geq T}\right).$$
	The set $D^F_\varphi$ is called the set of \emph{asymptotic directions} of $\varphi$. 
\end{definition}

Note that $D^F_\varphi$ depends neither on the choice of $B$-bounded curves used to close $\gamma_{x,t}$, nor on the choice of cover $B$ (by open, connected and contractible subsets). In fact, $D^F_\varphi$ does not depend on the choice of metric compatible with the topology of~$M$.
The set $D^F_\varphi$ was defined by Fried \cite{Fried82} to classify global section, but using curves of bounded length instead of $B$-bounded curves. We let the reader check that the two definitions coincide in Fried's setting.

Sullivan studied a similar set, which we defined below.
Assume temporally that $\varphi$ is of class $\Class^1$. Let $\mu$ be a $\varphi$-invariant probability measure. There exists a unique element in $H_1(M,\RR)$, denoted by $[\mu]_\varphi$, which satisfies that for all closed 1-form $\alpha$ on~$M$, we have 
\begin{equation}\label{eq-cohom-class}
	\alpha([\mu]_\varphi)=\int_M\alpha\left(\tfrac{\partial\varphi_t}{\partial t}\right)d\mu.
\end{equation}
When $\varphi$ is only continuous, the right side is undefined, so we can not take this equation as a definition. So let us define $[\mu]_\varphi$ in the general case.

Denote $\wh\varphi$ the lifted flow on $\wh M$. Take a continuous map $h\colon\wh M\to H_1(M,\RR)$ which we assume to be equivariant under the actions of $H_1(M,\ZZ)$ on the two sides. Such a function can be obtained as follows. Given a basis $(\alpha_1\cdots\alpha_n)$ of $H_1(M,\ZZ)$, and an $\alpha_i$-equivariant function $f_i\colon\wh M\to\RR$ (that can be build from a map $M\to\bfrac{\RR}{\ZZ}$ cohomologous to $\alpha_i$) for all $i$, the map $h=\sum_if_i\alpha_i$ is $H_1(M,\ZZ)$-equivariant. For any $t$ in $\RR$, the map $h\circ\wh\varphi_t-h$ is invariant under the action of $H_1(M,\ZZ)$, so it can be factorized by a function $g_t\colon M\to H_1(M,\RR)$ with $g_t\circ\wh\pi=h\circ\wh\varphi_t-h$.

\begin{lemma}
	For any $\varphi$-invariant signed measure $\mu$, the integral $\int_Mg_td\mu$ is linear in $t$. It is additionally independent on the choice of $h$ (among all $H_1(M,\ZZ)$-equivariant continuous functions). 
\end{lemma}

\begin{proof}
	For any $t,s$ in $\RR$, we have
	\begin{align*}
		h\circ\wh\varphi_{t+s}-h
			&=(h\circ\wh\varphi_{t+s}-h\circ\wh\varphi_{t})+(h\circ\wh\varphi_{t}-h) \\
			&=g_s\circ\varphi_t\circ\wh\pi + g_t\circ\wh\pi		
	\end{align*}
	so we have $g_{t+s}=g_s\circ\varphi_t+g_t$. It immediately follows that 
	$$\int_Mg_{t+s}d\mu=\int_Mg_s\circ\varphi_t d\mu+\int_Mg_td\mu=\int_Mg_sd\mu+\int_Mg_td\mu.$$

	To prove that it is independent on the choice of $h$, take $h'\colon\wh M\to H_1(M,\RR)$ another continuous and $H_1(M,\RR)$-equivariant function. Then $h'-h$ is invariant by the $H_1(M,\ZZ)$ action. Thus, it can be factorizes by a function $f\colon M\to H_1(M,\RR)$, which satisfies $f\circ\wt\pi=h'-h$. Denote by $g_t'$ the map associated to $h'$. Then we have $g_t'-g_t=f\circ\varphi_t-f$, which immediately implies $\int_M(g_t'-g_t)d\mu=0$.
\end{proof}

We define the homology class of $\mu$, relatively to $\varphi$, by $[\mu]_\varphi=\int_Mg_1d\mu$ in $H_1(M,\RR)$. It is clearly linear and continuous in $\mu$. We let the reader verify that when $\varphi$ is of class $\Class^1$, Equation~\ref{eq-cohom-class} is satisfied. Denote by $\MM_p(M)$ the set of probability measure on~$M$, equipped with the weak convergence topology. Also denote by $\MM_p(\varphi)$ the set of $\varphi$-invariant probability measures.

\begin{definition}
	Let $D^S_\varphi\subset H_1(M,\RR)$, $S$ for Sullivan, be the convex set:
	$$D^S_\varphi=\{[\mu]_\varphi\in H_1(M,\RR),\mu\in\MM_p(\varphi)\}$$
\end{definition}

This set is defined in the manner of Sullivan \cite{Sullivan1976}, though Sullivan did not used this definition.
Recall that $D_\varphi$ is the cone of asymptotic pseudo-directions, defined in Section~\ref{sec-ass-ps-dir}.

\begin{theorem}\label{thm-equiv-ass-dir}
	The three sets $\conv(D_\varphi)$, $\conv(D^F_\varphi)$ and $D^S_\varphi$ are equal. Additionally, if $\mu$ in $\MM_p(\varphi)$ is ergodic, then $[\mu]_\varphi$ lies in $D_\varphi\cap D^F_\varphi$.
\end{theorem}

We prove the theorem in several steps. The sets $D^F_\varphi$ and $D_\varphi$ are difficult to compare directly, so we use the third set as an intermediate step.

Given~$x$ in~$M$ and $t_1<t_2$, let $\gamma$ be a parametrization of $\varphi_{[t_1,t_2]}(x)$ given by $\gamma(s)=\varphi_s(x)$ for $s$ in $[t_1,t_2]$. We denote by $\Leb(x,t_1,t_2)$, the measure on~$M$ obtained as the push-back by $\gamma$ of the Lebesgue measure on $[t_1,t_2]$. It is supported on the orbit arc $\varphi_{[t_1,t_2]}(x)$

\begin{proof}[Proof of $\conv(D^F_\varphi)\subset D^S_\varphi$]
	Take two sequences $x_n$ in~$M$ and $t_n>0$ that goes to $+\infty$. We claim that the sequence of measure $\tfrac{1}{t_n}\Leb(x_n,0,t_n)$ accumulates on $\MM_p(\varphi)$ for the weak topology. Recall that $\MM_p(M)$ is compact for the weak topology, since~$M$ is compact. So up to an extraction, we may assume that the sequence converges toward some $\mu$ in $\MM_p(M)$. Let us take a continuous function $f\colon M\to\RR$ and compute:
	\begin{align*}
		\left|\int_M(f\circ\varphi_s-f)d\left(\tfrac{1}{t_n}\Leb(x_n,0,t_n)\right)\right|
			&= \tfrac{1}{t_n}\left|\int_0^{t_n}(f\circ\varphi_s-f)\circ\varphi_tdt\right| \\
			&= \tfrac{1}{t_n}\left|\int_0^{t_n}f\circ\varphi_{t+s}dt-\int_0^{t_n}f\circ\varphi_tdt\right| \\
			&= \tfrac{1}{t_n}\left|\int_s^{t_n+s}f\circ\varphi_tdt-\int_0^{t_n}f\circ\varphi_tdt\right| \\
			&= \tfrac{1}{t_n}\left|\int_{t_n}^{t_n+s}f\circ\varphi_tdt-\int_0^sf\circ\varphi_tdt\right| \\
			&\leq \tfrac{2s}{t_n}\|f\|_\infty\xrightarrow[n\to+\infty]{}0 \\
	\end{align*}
	Hence at the limit in~$n$, we obtain $\int_M(f\circ\varphi_s-f)d\mu=0$, so $\mu$ is $\varphi$-invariant.

	Recall that $\gamma_{x_n,t_n}$ is the concatenation of $\varphi_{[0,t_n]}(x_n)$ with a $B$-curve. We claim that the homology class of $\tfrac{1}{t_n}\gamma_{x_n,t_n}$ converges toward $[\mu]_\varphi$. To prove it, take $h\colon\wh M\to H_1(M,\RR)$ and $g_t\colon M\to H_1(M,\RR)$ as above, and take a lift $\wh x_n$ of $x_n$ in $\wh M$. Then we have:
	\begin{align}\label{eq-measure-approx}
		\frac{1}{t_n}\int_Mg_1d\Leb(x_n,0,t_n)
			&= \frac{1}{t_n}\int_0^{t_n}g_1\circ\varphi_s(x_n)ds \\
			&= \frac{1}{t_n}\int_0^{t_n}(h\circ\varphi_1-h)\circ\wh\varphi_s(\wh x_n)ds \\
			&= \frac{1}{t_n}\int_{t_n}^{t_n+1}h\circ\wh\varphi_s(\wh x_n)ds - \frac{1}{t_n}\int_0^1h\circ\wh\varphi_s(\wh x_n)ds
	\end{align}
	
	The left term converges toward $[\mu]_\varphi$. 
	For any $t_n\leq s\leq t_n+1$, the point $\wh\varphi_s(\wh x_n)$ remains at bounded distance from the point $[\gamma_{x_n,t_n}]\cdot\wh x_n$. So the integral $\int_{t_n}^{t_n+1}h\circ\wh\varphi_s(\wh x_n)ds$ remains at bounded distance from $[\gamma_{x_n,t_n}]+h(\wh x_n)$. Similarly, $\int_{0}^{1}h\circ\wh\varphi_s(\wh x_n)ds$ remains at bounded distance from $h(\wh x_n)$. So the right term get closer and closer to $\tfrac{1}{t_n}[\gamma_{x_n,t_n}]$ when~$n$ goes to $+\infty$.
	It follows that $\tfrac{1}{t_n}[\gamma_{x_n,t_n}]$ converges toward $[\mu]_\varphi$. 
	Therefore, we have $D^F_\varphi\subset D^S_\varphi$. The conclusion follows.	
\end{proof}

We prove the converse inclusion. We actually prove a stronger result:

\begin{lemma}\label{lem-ergodic-approx}
	Let $\mu$ be an ergodic measure. Then for $\mu$-almost all~$x$ in~$M$, the homology class $\tfrac{1}{t}[\gamma_{x,t}]$ converges toward $[\mu]_\varphi$ when $t$ goes to $+\infty$.
\end{lemma}

\begin{proof}[Proof of the lemma and of $D^S_\varphi\subset\conv(D^F_\varphi)$]
	Let $\mu$ be an ergodic $\varphi$-invariant probability measure. 
	Denote by $h\colon\wh M\to H_1(M,\RR)$ and $g_1\colon M\to H_1(M,\RR)$ the map used above. Then it follows from Birkhoff ergodic theorem that the integral
	$$I_t=\int_Mg_1d\left(\tfrac{1}{t}\Leb(x,0,t)\right)$$
	converges toward $\int_Mg_1d\mu=[\mu]_\varphi$ for $\mu$-almost all~$x$ in~$M$. Using Equation~\ref{eq-measure-approx} and the above arguments, the integral $I_t$ get closer to $\tfrac{1}{t}[\gamma_{x,t}]$ when $t$ goes to $+\infty$, so $\tfrac{1}{t}[\gamma_{x,t}]$ converges in $t$ toward $[\mu]_\varphi$ for $\mu$-almost all~$x$. Thus, $[\mu]_\varphi$ belongs to $D^F_\varphi$. Since the set $\MM_p(\varphi)$ is the closure of the convex hull of the ergodic measures, it follows that $D^S_\varphi\subset\conv(D^F_\varphi)$ holds.
\end{proof}

Given a periodic pseudo-orbit $\gamma\colon\bfrac{\RR}{l\ZZ}\to M$, we denote by $\Leb_\gamma=\gamma^*\Leb$ the push-back measure of the Lebesgue measure on $\bfrac{\RR}{l\ZZ}$. 
The proof of $\conv(D_\varphi)\subset D^S_\varphi$ is very similar to the one of $\conv(D^F_\varphi)\subset D^S_\varphi$.

\begin{proof}[Proof of $\conv(D_\varphi)\subset D^S_\varphi$]
	Let $\gamma$ be a periodic~$\epsilon$-pseudo-orbit. 
	We will prove that $\tfrac{1}{\len(\gamma)}[\gamma]$ accumulates on $\MM_p(\varphi)$ as~$\epsilon$ goes to zero. It is clear when $\gamma$ is taken among periodic orbits. 
	
	We treat the case when $\gamma$ is not continuous.
	Denote by $\Delta$ its discontinuity set. For any $t$ in $\Delta$, we write $x_t=\gamma(t)$. Let us denote by $l_t$ the length of the continuity interval of $\gamma$ that comes immediately after $t$, and denote by $I_t=[t,t+l_t[$ that interval. Note that $\Leb_\gamma$ is equal to the sum of the measures $\Leb(x_t,t,t+l_t)$ for all $t$ in $\Delta$. Take $s$ in $\RR$ and let us compute:
	\begin{align*}
		\tfrac{1}{\len(\gamma)}\left(\varphi_s^*\Leb_{\gamma}-\Leb_{\gamma}\right)
			& = \tfrac{1}{\len(\gamma)}\left(\sum_{t\in\Delta}\Leb(x_t,s,s+l_t)-\sum_{t\in\Delta}\Leb(x_t,0,l_t)\right) \\  
			& = \tfrac{1}{\len(\gamma)}\left(\sum_{t\in\Delta}\Leb(x_t,l_t,s+l_t)-\sum_{t\in\Delta}\Leb(x_t,0,s)\right) \\  
			& = \tfrac{1}{\len(\gamma)}\left(\sum_{t\in\Delta}\Leb(\varphi_{l_t}(x_t),0,s)-\sum_{t\in\Delta}\Leb(x_{t+l_t},0,s)  \right)
	\end{align*}
	where in the last line, we use the reparametrization $t\mapsto t+l_t$ of $\Delta$ on the right most term.
	As~$\epsilon$ goes to zero, $\varphi_{l_t}(x_t)$ converges uniformly toward $x_{t+l_t}$. 
	Note that at fixed $s$, the measure $\Leb(y,0,s)-\Leb(x,0,s)$ converges uniformly toward zero when the distance between~$x$ and $y$ goes to zero.
	Therefore, the measure $\Leb(\varphi_{l_t}(x_t),0,s)-\Leb(x_{t+l_t},0,s)$ converges uniformly toward zero when~$\epsilon$ goes to zero.
	Therefore, any accumulation point of the measures of the type $\tfrac{1}{\len(\gamma)}\Leb_{\gamma}$, when~$\epsilon$ goes to zero, is a $\varphi$-invariant measure. 
	
	For all $n\geq 1$, let $\gamma_n$ be a $\tfrac{1}{n}$-pseudo-orbit. Up to an extraction, we can assume that $\tfrac{1}{\len(\gamma_n)}\Leb_{\gamma_n}$ converges toward some $\mu$ in $\MM_p(\varphi)$. 
	Using Equation~\ref{eq-measure-approx}, and the same arguments that we used before, yields that $\tfrac{1}{\len(\gamma_n)}[\gamma_n]$ converges toward~$[\mu]_\varphi$. So $D_\varphi\subset D^S_\varphi$ holds true. The conclusion follows.
\end{proof}

\begin{proof}[Proof of $D^S_\varphi\subset\conv(D_\varphi)$]
	Take an ergodic invariant probability measure~$\mu$. Note that $\mu(M\setminus\Rec)$ is equal to zero. Indeed, $M\setminus\Rec$ is included in the wandering set: the set of point~$x$ in~$M$ which admit a neighborhood $U$ so that $\varphi_t(U)\cap U$ is empty for all $t$ large enough. And the wandering set has measure zero for any $\varphi$-invariant probability measure. The recurrence chains of $\varphi$ are disjoint and $\varphi$-invariant. So by ergodicity, all but one recurrence chain have measure zero of $\mu$. Thus, the support of $\mu$ is included in one recurrence chain~$R$~of~$\varphi$.
	
	The recurrence chain~$R$ is compact. Let us fix some small $\epsilon>0$ and take finitely many points $p_1\cdots p_r$ in~$R$, whose~$\epsilon$-neighborhoods cover~$R$. We may additionally assume that the~$\epsilon$-neighborhood of $p_i$ is included in one set of the cover $B$. For every pair $(p_i,p_j)$, take an~$\epsilon$-pseudo-orbits $c_{i,j}$ from $p_i$ to~$p_j$. We can assume that the length of the last orbit arc in~$c_{i,j}$ is also at least $T$. For instance, take a pseudo-orbit from $p_i$ to $\varphi_{-T}(p_j)$ and concatenate the orbit arc $\varphi_{[-T,0]}(p_j)$). Let $L$ be the largest length of the pseudo-orbits~$c_{i,j}$. 
	Denote by $E\subset H_1(M,\RR)$ the set of the homology classes of the curves obtained by closing any curve $c_{i,j}$ with any $B$-bounded curve from $p_j$ to $p_i$. Clearly $E$ is finite.

	Using Lemma~\ref{lem-ergodic-approx}, there exists~$x$ in~$R$ for which the homology class $\tfrac{1}{t}[\gamma_{x,t}]$ converges toward $[\mu]_\varphi$ when $t$ goes to $+\infty$. Recall that $\gamma_{x,t}$ is the orbit arc $\varphi_{[0,t]}(x)$ concatenated with a $B$-bounded curve.

	Take some $t>0$, and $i,j$ that satisfy $d(p_i, x)<\epsilon$ and $d(p_j, \varphi(t,x))<\epsilon$. The concatenation of $\varphi_{[0,t]}(x)$ and $c_{i,j}$ yields a periodic~$\epsilon$-pseudo-orbit, which we denote by $\delta_t$. It satisfies $t\leq\len(\delta_t)\leq t + L$ and $[\gamma_{x,t}]-[\delta_t]$ lies in $E$. So when $t$ goes to $+\infty$, the distance between $\tfrac{1}{\len(\delta_t)}[\delta_t]$ and $\tfrac{1}{t}[\gamma_{x,t}]$ goes to zero. It follows that $\tfrac{1}{\len(\delta_t)}[\delta_t]$ converges toward $[\mu]_\varphi$ when $t$ goes to~$+\infty$. It implies that $[\mu]_\varphi$ lies in $D_\varphi$. Therefore  $D^S_\varphi\subset\conv(D_\varphi)$ holds true.
\end{proof}

\addcontentsline{toc}{section}{References}
\bibliographystyle{alpha}
\bibliography{ref}

\end{document}